\documentclass{amsart}
\usepackage{amssymb}
\usepackage{amsmath}
\usepackage{latexsym}
\usepackage{eepic}
\usepackage{epsfig}
\usepackage{graphicx}
\usepackage{pb-diagram,lamsarrow,pb-lams,amscd}
\usepackage{eufrak}
\usepackage{mathrsfs}
\usepackage[dutch,english]{babel}
\usepackage[all,cmtip]{xy}

\DeclareMathAlphabet{\mathpzc}{OT1}{pzc}{m}{it}
\newtheorem{theorem}{Theorem}[section]
\newtheorem{theorem-definition}[theorem]{Theorem-Definition}
\newtheorem{lemma-definition}[theorem]{Lemma-Definition}
\newtheorem{definition-prop}[theorem]{Proposition-Definition}

\newtheorem{prop}[theorem]{Proposition}
\newtheorem{lemma}[theorem]{Lemma}
\newtheorem{cor}[theorem]{Corollary}
\newtheorem{definition}[theorem]{Definition}

\newtheorem{conjecture}[theorem]{Conjecture}
\newtheorem{example}[theorem]{Example}

\newenvironment{remark}{\vspace{4pt}\noindent\textbf{Remark.}}{\qed\vspace{4pt}}

\newcommand{\LL}{\ensuremath{\mathbb{L}}}

\newcommand{\N}{\ensuremath{\mathbb{N}}}
\newcommand{\Z}{\ensuremath{\mathbb{Z}}}
\newcommand{\Q}{\ensuremath{\mathbb{Q}}}

\newcommand{\C}{\ensuremath{\mathbb{C}}}

\newcommand{\A}{\ensuremath{\mathbb{A}}}

\newcommand{\X}{\ensuremath{\mathscr{X}}}

\newcommand{\mY}{\ensuremath{\mathfrak{Y}}}

\renewcommand{\C}{\ensuremath{\mathbb{C}}}

\renewcommand{\A}{\ensuremath{\mathbb{A}}}
\renewcommand{\X}{\ensuremath{\mathfrak{X}}}

\renewcommand{\mY}{\ensuremath{\mathfrak{Y}}}
\newcommand{\mZ}{\ensuremath{\mathfrak{Z}}}

\newcommand{\mV}{\ensuremath{\mathfrak{V}}}

\newcommand{\Spec}{\ensuremath{\mathrm{Spec}\,}}
\newcommand{\Spf}{\ensuremath{\mathrm{Spf}\,}}

\numberwithin{equation}{section}

\begin{document}
\title[A trace formula for varieties]{A trace formula for varieties over a discretely valued field}
\author[Johannes Nicaise]{Johannes Nicaise}
\address{Universit\'e Lille 1\\
Laboratoire Painlev\'e, CNRS - UMR 8524\\ Cit\'e Scientifique\\59655 Villeneuve d'Ascq C\'edex \\
France} \email{johannes.nicaise@math.univ-lille1.fr}
\begin{abstract}
We study the motivic Serre invariant of a smoothly bounded
algebraic or rigid variety $X$ over a complete discretely valued
field $K$ with perfect residue field $k$. If $K$ has
characteristic zero, we extend the definition to arbitrary
$K$-varieties using Bittner's presentation of the Grothendieck
ring and a process of N\'eron smoothening of pairs of varieties.

The motivic Serre invariant can be considered as a measure for the
set of unramified points on $X$. Under certain tameness
conditions, it admits a cohomological interpretation by means of a
trace formula. In the curve case, we use T. Saito's geometric
criterion for cohomological tameness to obtain more detailed
results. We discuss some applications to Weil-Ch\^atelet groups,
Chow motives, and the structure of the Grothendieck ring of
varieties.

\end{abstract}
 \maketitle
\section{Introduction}
Let $R$ be a complete discrete valuation ring, with quotient field
$K$ and perfect residue field $k$. If $X$ is a smoothly bounded
rigid $K$-variety (e.g. smooth, quasi-compact and separated), then
one can associate to $X$ its so-called motivic Serre invariant
$S(X)$. If $\X/R$ is a formal weak N\'eron model for $\X$, then
$S(X)$ is the class of the special fiber $\X_s$ in the quotient of
the Grothendieck ring of $k$-varieties modulo the ideal generated
by the class of the torus $\mathbb{G}_{m,k}$. Of course, one has
to show that this value only depends on $X$ and not on the choice
of a weak N\'eron model. If $X$ is smooth and quasi-compact, this
was proven in \cite{motrigid} using the theory of motivic
integration on formal $R$-schemes \cite{sebag1}, and the general
case can be deduced from this result.

By definition, the generic fiber $\X_\eta$ of $\X$ is an open
rigid subvariety of $X$ which contains all the unramified points
on $X$. Since $\X$ is smooth over $R$, its special fiber is a good
measure for the set of unramified points on $\X_\eta$. Therefore,
one can consider the motivic Serre invariant $S(X)$ as a measure
for the set of unramified points on $X$. It is natural to ask if
this invariant admits a cohomological interpretation in terms of
the Galois action on the \'etale cohomology of $X$. This is indeed
the case: under certain finiteness and tameness conditions on $X$,
a trace formula expresses the Euler characteristic of $S(X)$ in
terms of the trace of a monodromy operator on the tame $\ell$-adic
cohomology of $X$ \cite[6.4]{Ni-trace}.

The main themes of the present article are the following:
\begin{enumerate}
\item \label{theme-error} study of the error term in the trace
formula in the non-tame case, \item \label{theme-gen}
generalization of the definition of the motivic Serre invariant to
arbitrary algebraic $K$-varieties, if $K$ has characteristic zero,
\item \label{theme-groth} realization morphisms and structure of
the Grothendieck ring of varieties.
\end{enumerate}
These themes are tightly interwoven.

In Section \ref{sec-groth} we recall the definition of the
Grothendieck ring $K_0(Var_k)$ of varieties over an arbitrary
field $k$, and its localization $\mathcal{M}_k$ w.r.t. the class
$\LL$ of the affine line. This ring is still poorly understood,
and one of the main problems is to decide when two $k$-varieties
$X$ and $Y$ define the same class in $K_0(Var_k)$ or
$\mathcal{M}_k$. To this end, it is important to have some
``computable'' realization morphisms from $K_0(Var_k)$ to more
concrete rings. If $k$ has characteristic zero, the main tools are
the theorems of Bittner (Theorem \ref{bittner}) and Larsen and
Lunts (Theorem \ref{stabir}). Both use resolution of singularities
and weak factorization, which explains the restriction on the
characteristic. These theorems imply the existence of some
fundamental realization morphisms: stably birational classes
(Theorem \ref{stabir}), Albanese (Corollary \ref{albanese}), Chow
motives (Theorem \ref{chow}). Larsen and Lunts' Theorem gives a
beautiful description of the Grothendieck ring modulo the ideal
generated by the class of the affine line, but it tells nothing
about the localized Grothendieck ring $\mathcal{M}_k$ (the same
holds for the Albanese realization).

In positive characteristic, we are considerably less equiped. In
Proposition \ref{spread} we formulate the classical technique of
``spreading out'' on the level of Grothendieck rings. This tool
allows to reduce questions about $K_0(Var_k)$ and $\mathcal{M}_k$
to a finitely generated base ring.
As an application, we define the Poincar\'e polynomial for
arbitrary separated morphisms of finite type of schemes (Section
\ref{sec-poingen} and Section \ref{Appendix}). Roughly speaking,
it is the only constructible invariant which is compatible with
base change and gives the correct result over a finite field (viz.
the polynomial whose coefficients are given by the virtual Betti
numbers, which are defined in terms of the weight filtration on
$\ell$-adic cohomology). We refer to Theorem \ref{thm-poin} for
the exact statement. The Poincar\'e polynomial defines ring
morphisms $K_0(Var_k)\rightarrow \Z[T]$ and
$\mathcal{M}_k\rightarrow \Z[T,T^{-1}]$ for an arbitrary field
$k$, which provide a new way to distinguish elements in these
Grothendieck rings (Proposition \ref{nonzero}).

Next, we turn our attention to the motivic Serre invariant. Let
$R$ be a complete discrete valuation ring, with quotient field $K$
and perfect residue field $k$. If $X$ is a smooth and proper
$K$-variety, then the associated rigid $K$-variety $X^{an}$ is
separated, smooth and quasi-compact, so $S(X):=S(X^{an})\in
K_0(Var_k)/(\LL-1)$ is well-defined. Our main result in this
setting (Theorem \ref{serresing}) states that if $K$ has
characteristic zero, this invariant can be uniquely extended to an
additive and multiplicative invariant on the category of
$K$-varieties, i.e. a ring morphism $S:\mathcal{M}_K\rightarrow
K_0(Var_k)/(\LL-1)$. This ring morphism is interesting for two
reasons: it defines the motivic Serre invariant for arbitrary
$K$-varieties, and it provides a new realization of
$\mathcal{M}_K$ which is computable in significant cases (see e.g.
Theorem \ref{elliptic}). We use it in Proposition
\ref{not-injective} to show that the realization morphisms to
(effective and non-effective) Chow motives are not injective.

The existence of the morphism $S$ can be deduced from Bittner's
theorem once we understand how the motivic Serre invariant behaves
w.r.t. the blow-up relations. To this end, we extend N\'eron's
smoothening process to pairs of smooth $K$-varieties in Section
\ref{sec-smooth} (Theorem \ref{smoothen}). This result implies the
existence of weak N\'eron models for bounded and smooth pairs
(Definition \ref{def-weakner}). Theorem \ref{smoothen} is proved
by using the detailed information on the centers of the blow-ups
in the classical smoothening process \cite[3.5.2]{neron} and
Greenberg's Theorem \cite{Gr}.

In Section \ref{sec-bound} we study and compare different notions
of boundedness for rigid and algebraic varieties. In particular,
we show that, if $K$ has characteristic zero, a smooth $K$-variety
$X$ is bounded iff it has a compactification without unramified
points at infinity; then this holds for every smooth
compactification (Proposition \ref{nopoints}). This result is used
to prove that the motivic Serre invariant of a $K$-variety without
unramified points is zero (Proposition \ref{empty}).

Section \ref{sec-trace} deals with the trace formula in the tame
case. We assume that the residue field $k$ of $R$ is algebraically
closed, and for each $d>0$ prime to the characteristic exponent
$p$ of $k$, we denote by $K(d)$ the unique degree $d$ extension of
$K$ inside a fixed algebraic closure. If $X$ is a tame smooth and
proper $K$-variety (in the sense of Definition \ref{def-tame}) and
$d>0$ is an integer prime to $p$, then the trace formula
(Proposition \ref{tame-3}) states that the Euler characteristic of
the motivic Serre invariant of the $K(d)$-variety $X\times_K K(d)$
equals the trace of a generator of the tame monodromy group
$G(K^t/K(d))$ on the tame $\ell$-adic cohomology of $X$. If $K$
has characteristic zero, then by formal arguments, this result
extends to any $K$-variety whose isomorphism class belongs to the
subring of $K_0(Var_K)$ generated by the classes of tame smooth
and proper varieties (Theorem \ref{trace}). In particular, if $k$
has characteristic zero, the trace formula holds for any
$K$-variety. This yields a  sufficient and necessary cohomological
condition for the existence of a rational point (Corollary
\ref{existpoint}).

Without tameness conditions, the trace formula no longer holds. We
take a closer look at the case of curves in Section
\ref{sec-curves}. A computation on the nearby cycles yields an
expression for the error term in the trace formula in terms of the
geometry of a regular model with normal crossings (Theorem
\ref{general}). Intriguingly, this expression appears to be
related to Saito's geometric criterion for cohomological tameness
\cite{saito}, and this relation shows that the trace formula holds
for cohomologically tame curves of genus $\neq 1$ (Theorem
\ref{global}) and cohomologically tame elliptic curves (Theorem
\ref{elliptic}). In the case of elliptic curves we can make
explicit computations of the motivic Serre invariant and the error
term using the Kodaira-N\'eron reduction table.

 The case of genus $1$ curves without rational point brings some surprises. Even if the wild
ramification acts trivially, the trace formula can fail
(Proposition \ref{counterex}), and more fundamentally, the motivic
Serre invariant does not admit any general cohomological (nor even
motivic) interpretation (at least if we work with rational
coefficients). The cause is the fact that there are
cohomologically tame elliptic curves $E$ over $K$ with non-trivial
Weil-Ch\^atelet group whose motivic Serre invariant has non-zero
Euler characteristic. If $X$ is a non-trivial $E$-torsor then the
Chow motives of $E$ and $X$ are isomorphic (and hence their
$\ell$-adic cohomology spaces are isomorphic as Galois modules),
but their motivic Serre invariants have distinct Euler
characteristics ($S(X)=0$ since $X$ has no $K$-rational point).
Therefore, the trace formula can not hold for both $X$ and $E$
(over finite fields this situation cannot occur since the
Weil-Ch\^atelet group of an elliptic curve over a finite field is
zero). Reversing the arguments, we can use the validity of the
trace formula in certain cases to recover triviality results about
the Weil-Ch\^atelet group (Proposition \ref{h1}). Finally, we use
the local version of Saito's criterion to prove a trace formula
for the analytic Milnor fiber (Theorem \ref{local}).
\subsection*{Notation}
We denote by $(Sch)$ the category of schemes. For any scheme $S$,
we denote by $(Sch/S)$ the category of schemes over $S$. If
$S=\Spec A$ is affine, we write also $(Sch/A)$ instead of
$(Sch/S)$. If $x$ is a point on $S$, we will denote by $k(x)$ its
residue field. We write $S^o$ for the set of closed points on $S$.

We denote by
$$(.)_{red}:(Sch/S)\rightarrow (Sch/S):X\mapsto X_{red}$$ the functor
mapping a $S$-scheme $X$ to its maximal reduced closed subscheme
$X_{red}$.

For any scheme $S$, a $S$-variety is a reduced separated
$S$-scheme of finite type. For any separated $S$-scheme of finite
type $X$, we denote by $Sm(X)$ the open subscheme of $X$
consisting of the points where $X$ is smooth over $S$. If we want
to make the base scheme $S$ explicit, we'll write $Sm(X/S)$
instead of $Sm(X)$.

For any field $F$, we denote by $F^a$ an algebraic closure, and by
$F^s$ the separable closure of $F$ in $F^a$. Starting from Section
\ref{sec-smooth}, $R$ denotes a discrete valuation ring, with
quotient field $K$ and residue field $k$. The maximal ideal of $R$
will be denoted by $\mathfrak{M}$. We fix a separable closure
$K^s$ of $K$, and we denote by $R^{sh}$ the strict henselization
of $R$ in $K^s$, and by $K^{sh}\subset K^s$ its quotient field.
We denote by $k^s$ the residue field of $R^{sh}$. The field $k^s$
is a separable closure of $k$. Moreover, we denote by $K^t$ the
tame closure of $K$ inside $K^s$. We fix a prime $\ell$ invertible
in $k$. Additional assumptions will be indicated at the beginning
of each section.

If $R$ is complete, and we fix a value $0<\theta<1$, then there
exists a unique absolute value $|\cdot|$ on $K^s$ such that
$|a|=\theta^{v(a)}$ for each $a$ in $K^*$, where $v$ denotes the
discrete valuation on $K^*$. This absolute value makes $K$ into a
non-archimedean field.

We'll consider the generic fiber functor
$$(.)_K:(Sch/R)\rightarrow (Sch/K):X\mapsto X_K=X\times_R K$$
as well as the special fiber functor
$$(.)_s:(Sch/R)\rightarrow (Sch/k):X\mapsto X_K=X\times_R k$$

For any scheme or rigid variety $X$ and any prime $\ell$, we put
$$H(X,\Q_\ell)=\oplus_{i\geq 0}H^i(X,\Q_\ell)$$ where
$H^i(X,\Q_\ell)$ is the $i$-th $\ell$-adic cohomology space, and
we view $H(X,\Q_\ell)$ as a $\Z$-graded vector space. Similar
notation applies for cohomology with compact supports. If
$H=\oplus_{i\in \Z}H^i$ is a finite dimensional graded vector
space over some field $F$, and $M$ is a graded endomorphism of
$H$, then we put
$$Trace(M\,|\,H)=\sum_{i \in \Z}(-1)^i Trace(M\,|\,H^i)$$

\section{The Grothendieck ring of varieties}\label{sec-groth}
\subsection{Definition and realization morphisms}\label{sec-real}
\begin{definition}[Grothendieck ring]\label{def-grothring}
Let $S$ be any Noetherian scheme. We define the Grothendieck group
of $S$-varieties $K_0(Var_S)$ as the abelian group with
\begin{itemize}
\item generators: isomorphism classes $[X/S]$ of separated
$S$-schemes of finite type $X$ \item relations: if $Y\rightarrow
X$ is a closed immersion, then $[X/S]=[(X-Y)/S]+[Y/S]$ (``scissor
relations'').
\end{itemize}
The product $[X/S]\cdot[Y/S]=[(X\times_S Y)/S]$ defines a ring
structure on $K_0(Var_S)$, and we call this ring the Grothendieck
ring of $S$-varieties.

Moreover, we put $\LL_S=[\A^1_S/S]$ and
$\mathcal{M}_S=K_0(Var_S)[\LL^{-1}_S]$.
\end{definition}

A morphism of Noetherian schemes $T\rightarrow S$ induces base
change morphisms of rings $K_0(Var_S)\rightarrow K_0(Var_T)$ and
$\mathcal{M}_S\rightarrow \mathcal{M}_T$. Moreover, a separated
morphism of finite type $S\rightarrow U$ induces forgetful
morphisms of abelian groups $K_0(Var_S)\rightarrow K_0(Var_U)$ and
$\mathcal{M}_S\rightarrow \mathcal{M}_U$ (the definition of the
latter requires some care: if $X$ is a separated $S$-scheme of
finite type, then $[X/S]\LL_S^{-i}$ is mapped to
$[X/U]\LL_U^{-i}$, for any $i\in \Z$).

The following properties follow immediately from the definition:
for any separated $S$-scheme of finite type $X$, the natural
closed immersion $X_{red}\rightarrow X$ gives rise to the equality
$[X/S]=[X_{red}/S]$ in $K_0(Var_S)$. Likewise, the base change
morphisms $K_0(Var_S)\rightarrow K_0(Var_{S_{red}})$ and
$\mathcal{M}_S\rightarrow \mathcal{M}_{S_{red}}$ are ring
isomorphisms.

If the base scheme $S$ is clear from the context, we write $[X]$
and $\LL$ instead of $[X/S]$ and $\LL_S$. If $S$ is affine, say
$S=\Spec A$, then we write $K_0(Var_A)$ and $\mathcal{M}_A$
instead of $K_0(Var_S)$ and $\mathcal{M}_S$.

Even if $k$ is a field of characteristic zero, the Grothendieck
ring $K_0(Var_k)$ is not very well understood. Poonen showed that
$K_0(Var_k)$ is not a domain \cite{Poo}.
 It is not
known if the natural morphism $K_0(Var_k)\rightarrow
\mathcal{M}_k$ is injective (i.e. if $\LL$ is a zero divisor in
$K_0(Var_k)$). We refer to \cite{lalu-julienliu} for some
intriguing questions and results.

Now let $k$ be any field. By its definition, $K_0(Var_k)$ is the
universal additive and multiplicative invariant for the category
$Var_k$ of $k$-varieties: any invariant of $k$-varieties with
values in a ring $A$ which is additive w.r.t. closed immersions
and multiplicative w.r.t. the product of $k$-varieties, defines a
ring morphism $K_0(Var_k)\rightarrow A$. Here are some well-known
examples we will need:

\vspace{5pt}

 $(1)$ \textit{Counting rational points:} if $k$ is a
finite field, then there exists a unique ring morphism
$$\sharp:K_0(Var_k)\rightarrow \Z$$ which maps $[X]$ to $\sharp X(k)$ (the
number of $k$-rational points) for each separated $k$-scheme of
finite type $X$. It localizes to a ring morphism
$\sharp:\mathcal{M}_k\rightarrow \Q$.

$(2)$ \textit{Etale realization:} let $k$ be any field, and denote
by $G_k$ the absolute Galois group $G(k^s/k)$. We fix a prime
$\ell$ invertible in $k$, and we denote by $Rep_{G_k}(\Q_\ell)$
the abelian tensor category of $\ell$-adic representations of
$G_k$ (i.e. finite dimensional $\Q_\ell$-vector spaces with a
continuous left action of $G_k$). The tensor structure on
$Rep_{G_k}(\Q_\ell)$ defines a ring structure on the Grothendieck
group $K_0(Rep_{G_k}(\Q_\ell))$. As pointed out in \cite{naumann},
there exists a unique ring morphism
$$\acute{e}t:K_0(Var_k)\rightarrow K_0(Rep_{G_k}(\Q_\ell))$$ such that
$$\acute{e}t([X])=\sum_{i\geq 0}(-1)^i [H^i_c(X\times_k k^s,\Q_\ell)]$$
for each separated $k$-scheme of finite type $X$. It localizes to
a ring morphism $\acute{e}t:\mathcal{M}_k\rightarrow
K_0(Rep_{G_k}(\Q_\ell))$ (since $\acute{e}t(\LL)=[\Q_\ell(-1)]$ is
invertible in $K_0(Rep_{G_k}(\Q_\ell))$).

$(3)$ \textit{The $\ell$-adic Euler characteristic:} if $k$ is any
field and $\ell$ is a prime invertible in $k$, then there exists a
unique ring morphism
$$\chi_{top}:\mathcal{M}_k\rightarrow \Z$$ such that
$$\chi_{top}([X])=\sum_{i\geq 0}(-1)^i \mathrm{dim}\,H^i_c(X\times_k k^s,\Q_\ell)$$
for each separated $k$-scheme of finite type $X$. It can also be
obtained by composing the \'etale realization $\acute{e}t$ with
the forgetful morphism
$$K_0(Rep_{G_k}(\Q_\ell))\rightarrow K_0(\Q_\ell)=\Z$$ The
morphism $\chi_{top}$ is independent of $\ell$ (this is well
known: if $k$ has characteristic zero it follows from comparison
with singular cohomology; if $k$ is finite from the cohomological
interpretation of the zeta function; if $k$ is any field of
positive characteristic by spreading out and reduction to a finite
base field).

$(4)$ \textit{The Hodge-Deligne realization:}
 assume $k=\C$, and
define the Hodge-Deligne polynomial $HD(X;u,v)$ of a separated
$\C$-scheme of finite type $X$ by
$$HD(X;u,v)=\sum_{k^\geq 0}(-1)^k h^{p,q}(H^k_c(X(\C),\C))u^pv^q$$
where $h^{p,q}(H^k_c(X,\C))$ denotes the dimension of the
$(p,q)$-component of Deligne's mixed Hodge structure on
$H^k_c(X,\C)$. Then $HD(\cdot;u,v)$ is additive and
multiplicative, so there exists a unique ring morphism
$$HD:K_0(Var_{\C})\rightarrow \Z[u,v]$$ mapping $[X]$ to
$HD(X;u,v)$ for any separated $\C$-scheme of finite type $X$. It
localizes to a ring morphism $HD:\mathcal{M}_{\C}\rightarrow
\Z[u,v,u^{-1},v^{-1}]$.

The definition of $HD$ generalizes to an arbitrary base field $k$
of characteristic zero, either by invoking the Lefschetz principle
(the Hodge numbers are algebraic invariants) or by using Bittner's
presentation of the Grothendieck ring (Theorem \ref{bittner}).

\vspace{5pt}
 As a general rule, whenever $\mu$ is a group morphism
from $K_0(Var_k)$ or $\mathcal{M}_k$ to some abelian group $A$, we
write $\mu(X)$ instead of $\mu([X])$ for any separated $k$-scheme
of finite type $X$.
\subsection{Bittner's presentation and the theorem of Larsen and Lunts}
Let $k$ be any field.
\begin{definition}
We denote by $K^{(bl)}_0(Var_k)$ the abelian group given by
\begin{itemize}
\item generators: isomorphism classes $[X]_{bl}$ of smooth,
projective $k$-varieties $X$ \item relations:
$[\emptyset]_{bl}=0$, and if $Y$ is a closed subvariety of $X$,
smooth over $k$, $X'\rightarrow X$ is the blow-up of $X$ with
center $Y$, and $E=X'\times_X Y$ is the exceptional divisor, then
$[X']_{bl}-[E]_{bl}=[X]_{bl}-[Y]_{bl}$ (``blow-up relations'').
\end{itemize}
The product $[X]_{bl}\cdot[Y]_{bl}=[X\times_k Y]_{bl}$ defines a
ring structure on $K^{(bl)}_0(Var_k)$.

The ring $K_0^{(bl)}(Var_k)'$ is defined in the same way,
replacing ``projective'' by ``proper''.
\end{definition}
Note that the product is well-defined, since blow-ups commute with
flat base change. It follows immediately from the definition that
there exist unique ring morphisms
\begin{eqnarray*}
\alpha&:&K^{(bl)}_0(Var_k)\rightarrow K_0(Var_k)
\\ \alpha'&:&K^{(bl)}_0(Var_k)'\rightarrow K_0(Var_k)
\end{eqnarray*} mapping $[X]_{bl}$ to $[X]$ for any smooth, projective (resp. proper) $k$-variety $X$.
\begin{theorem}[Bittner \cite{Bitt2},\,Thm.\,3.1]\label{bittner}
If $k$ has characteristic zero, then the natural ring morphisms
\begin{eqnarray*}
\alpha&:&K^{(bl)}_0(Var_k)\rightarrow K_0(Var_k)
\\ \alpha'&:&K^{(bl)}_0(Var_k)'\rightarrow K_0(Var_k)
\end{eqnarray*}
are isomorphisms.
\end{theorem}
It follows easily from Hironaka's resolution of singularities that
$\alpha$ and $\alpha'$ are surjective. Using Weak Factorization
\cite{weakfact}, Bittner also proved injectivity.



Recall that two smooth, projective, connected $k$-varieties
$X$,\,$Y$ are called stably birational if there exist integers
$m,n\geq 0$ such that $X\times_k\mathbb{P}^m_k$ and $Y\times_k
\mathbb{P}^n_k$ are birational. This defines an equivalence
relation on the set of smooth, projective, connected
$k$-varieties. We denote by $SB$ the set of equivalence classes
and by $\Z[SB]$ the free abelian group on $SB$.
\begin{theorem}[Stably birational realization
\cite{lars}]\label{stabir} If $k$ has characteristic zero, then
there exists a unique isomorphism of abelian groups
$$\Phi_{SB}:K_0(Var_k)/\LL K_0(Var_k)\rightarrow \Z[SB]$$ mapping
a smooth, projective, connected $k$-variety to its equivalence
class in $SB$.
\end{theorem}
As explained in \cite[2.4+7]{lars} the existence of $\Phi_{SB}$
follows immediately from Theorem \ref{bittner}, and the fact that
it is an isomorphism follows easily from Weak Factorization
\cite{weakfact}. In \cite{lars} it was assumed that $k$ is
algebraically closed, but this is not necessary
\cite[p.\,28]{kollar}.
\begin{cor}[Albanese realization \cite{Poo} ]\label{albanese}
Assume that $k$ has characteristic zero, denote by $AV_k$ the
monoid of isomorphism classes of abelian varieties over $k$, and
by $\Z[AV_k]$ the associated monoid ring. There exists a unique
ring morphism
$$Alb:K_0(Var_k)\rightarrow \Z[AV_k]$$ which sends the class $[X]$
of a smooth, projective, connected $k$-variety $X$ to the
isomorphism class of its Albanese $Alb(X)$ in $\Z[AV_k]$.

In particular, if $A$, $B$ are abelian varieties over $k$, then
$[A]=[B]$ in $K_0(Var_k)$ iff $A\cong B$.
\end{cor}
\begin{proof}
The Albanese is invariant under stably birational equivalence.
\end{proof}
Note that $Alb(\LL)=0$, so that $Alb$ does not localize to a
realization of $\mathcal{M}_k$.

\subsection{Specialization to Chow motives}\label{sec-chow}
Let $k$ be any field, denote by $Mot^{eff}_{k}$ the category of
effective Chow motives over $k$ with rational coefficients, and by
$Mot_{k}$ the category of Chow motives with rational coefficients.
The natural functor
$$Mot^{eff}_{k}\rightarrow Mot_{k}$$ is additive and compatible with the tensor product, so it yields a
natural ring morphism
$$\rho:K_0(Mot^{eff}_{k})\rightarrow K_0(Mot_{k})$$
which induces an isomorphism
$$K_0(Mot^{eff}_k)([\LL_{mot}]^{-1})\cong K_0(Mot_{k})$$ where
$\LL_{mot}$ denotes the Lefschetz motive. I do not know if $\rho$
is injective (i.e. if $[\LL_{mot}]$ is not a zero divisor in
$K_0(Mot^{eff}_{k})$). Using the fact that
$Mot_{k}^{eff}\rightarrow Mot_k$ is fully faithful, it is easily
seen that $\rho$ is injective if one assumes the following
conjecture:
\begin{conjecture}[Goettsche \cite{goettsche}, Conj.
2.5]\label{Goettsche} If $M$ and $N$ are objects of $Mot^{eff}_k$,
then $[M]=[N]$ in $K_0(Mot^{eff}_k)$ iff $M$ and $N$ are
isomorphic.
\end{conjecture}
\begin{prop}
Assume that Conjecture \ref{Goettsche} holds. If $M$ and $N$ are
objects of $Mot_k$, then $[M]=[N]$ in $K_0(Mot_k)$ iff $M\cong N$
in $Mot_k$. Moreover, $\rho$ is injective.
\end{prop}
\begin{proof}
If $[M]=[N]$ in $K_0(Mot_k)$, then there exists an object $P$ in
$Mot_k$ such that $M\oplus P\cong N\oplus P$. For $i\gg 0$,
$P\otimes \LL_{mot}^i$, $M\otimes \LL_{mot}^i$ and
$N\otimes\LL_{mot}^i$ are effective, and then
$$(M\otimes \LL_{mot}^i)\oplus (P\otimes \LL_{mot}^i)\cong (N\otimes \LL_{mot}^i)\oplus (P\otimes
\LL_{mot}^i)$$ in $Mot_k^{eff}$. Hence, $[M\otimes
\LL^i_{mot}]=[N\otimes \LL^i_{mot}]$ in $K_0(Mot^{eff}_k)$ and by
Conjecture \ref{Goettsche}, this implies that $M\otimes
\LL^i_{mot}$ and $N\otimes \LL^i_{mot}$ are isomorphic in
$Mot^{eff}_k$. Tensoring with $\LL^{-i}_{mot}$ shows that $M\cong
N$ in $Mot_k$.

This easily implies the injectivity of $\rho$: any element
$\alpha$ of $K_0(Mot^{eff}_k)$ can be written as $[M]-[N]$ with
$M$ and $N$ objects of $Mot^{eff}_k$, and $\rho(\alpha)=0$ means
that $M$ and $N$ have the same class in $K_0(Mot_k)$. Hence, $M$
and $N$ are isomorphic in $Mot_k$, and therefore also in
$Mot^{eff}_k$, so $\alpha=0$.
\end{proof}

 It was shown in \cite{goettsche} that Conjecture 2.5 follows from the Beilinson-Murre Conjecture
\cite[Conj.\,2.1+5.1+Thm.\,5.2]{jannsen}

We denote, for each smooth and projective variety $X$ over $k$, by
$M(X)$ the motive $(X,id)$ associated to $X$ in $Mot^{eff}_{k}$.
With slight abuse of notation, we'll use the same notation for its
image $(X,id,0)$ in $Mot_{k}$.

%

 \begin{theorem}[Gillet-Soul\'e \cite{gillet-soule}, Guillen-Navarro Aznar \cite{GuiNav}, Bittner \cite{Bitt2}]\label{chow}
Assume that $k$ has characteristic zero. There exist unique ring
morphisms
\begin{eqnarray*}
\chi^{eff} &:& K_0(Var_k)  \rightarrow  K_0(Mot^{eff}_{k})
\\ \chi & : & \mathcal{M}_k  \rightarrow  K_0(Mot_{k})
\end{eqnarray*}
such that, for any smooth and projective $k$-variety $X$,
$\chi^{eff}(X)$ (resp. $\chi(X)$) is the class of $M(X)$ in
$K_0(Mot^{eff}_{k})$ (resp. $K_0(Mot_{k})$).
\end{theorem}

The question about the existence of such a morphism $\chi^{eff}$
was raised already by Grothendieck in a letter to Serre
\cite[letter of 16/8/1964]{corr-grothserre}; he also asked how far
the morphism $\chi^{eff}$ is from being bijective.
It is known that $\chi^{eff}$ is not injective: isogeneous abelian
varieties have isomorphic Chow motives with rational coefficients,
while, if $k$ has characteristic zero, the classes of two abelian
varieties in $K_0(Var_k)$ coincide iff the varieties are
isomorphic, because of the existence of the Albanese realization
(Corollary \ref{albanese}).

However, this example does not answer the following question:
\textit{Is $\chi$ injective?} It is not known if $\LL$ is a zero
divisor in $K_0(Var_k)$, and the Albanese realization $Alb$ maps
$\LL$ to zero, so it is not clear if two non-isomorphic abelian
varieties have distinct classes in $\mathcal{M}_k$.


We will show in Proposition \ref{not-injective} that, for an
appropriate base field $k$ of characteristic zero, $\chi$ is
non-injective.
I do not
know if $\chi$ and $\chi^{eff}$ are surjective.

\begin{remark}
Theorem \ref{chow} still holds if we replace rational coefficients
by integer coefficients \cite[Thm.\,4]{gillet-soule}. By Theorem
\ref{bittner}, we only have to check that Chow motives satisfy the
blow-up relations. For rational coefficients, this was proven in
\cite[5.1]{GuiNav}, but the same proof holds for $\Z$-coefficients
(see \cite[0.1.3]{beauv} for a computation of the Chow groups). If
we work with $\Z$-coefficients, I do not know if $\chi$ and
$\chi^{eff}$ are injective.
\end{remark}
\subsection{Spreading out}\label{sec-spread}
Let $k$ be any field. We denote by $\mathscr{A}_k$ the set of
finitely generated sub-$\Z$-algebras of $k$, ordered by inclusion.
Then $k$ is the limit of the direct system $\mathscr{A}_k$ in the
category of rings. If $X$ is a $k$-variety, and $A$ is an object
in $\mathscr{A}_k$, then a $A$-model for $X$ is a $A$-variety
$X_A$ endowed with an isomorphism $X\cong X_A \times_A k$. An
$A$-model for a morphism of $k$-varieties $f:X\rightarrow Y$ is a
morphism of $A$-varieties $f_A:X_A\rightarrow Y_A$ such that $X_A$
and $Y_A$ are $A$-models for $X$, resp. $Y$, and such that $f$
coincides with the morphism $X_A\times_A k\rightarrow Y_A\times_A
k$ obtained from $f_A$ by base change (modulo the identifications
$X_A\times_A k\cong X$ and $Y_A\times_A k\cong Y$).

For any pair of objects $A,A'$ in $\mathscr{A}_k$ with $A\subset
A'$, we consider the natural base change morphisms
\begin{eqnarray*}
\phi_{A}^{A'}&:&K_0(Var_{A})\rightarrow K_0(Var_{A'})
\\ \psi_{A}^{A'}&:&\mathcal{M}_A\rightarrow \mathcal{M}_{A'}
\end{eqnarray*}
as well as
\begin{eqnarray*}
\phi_{A}^{k}&:&K_0(Var_{A})\rightarrow K_0(Var_{k})
\\ \psi_{A}^{k}&:&\mathcal{M}_A\rightarrow \mathcal{M}_{k}
\end{eqnarray*}
We obtain direct systems of rings $(K_0(Var_A),\phi_A^{A'})$ and
$(\mathcal{M}_A,\psi_A^{A'})$ indexed by the directed set
$\mathscr{A}_k$, and the morphisms $\phi_{A}^{k}$ and
$\psi_{A}^{k}$ induce morphisms
\begin{eqnarray*}
\phi&:&\lim_{\stackrel{\longrightarrow}{A\in
\mathscr{A}_k}}K_0(Var_{A})\rightarrow K_0(Var_{k})
\\ \psi&:&\lim_{\stackrel{\longrightarrow}{A\in
\mathscr{A}_k}}\mathcal{M}_A\rightarrow \mathcal{M}_{k}
\end{eqnarray*}

The classical technique of ``spreading out'' can be formulated on
the level of Grothendieck rings in the following way.

\begin{prop}[Spreading out]\label{spread}
The natural ring morphisms
\begin{eqnarray*}
\phi&:&\lim_{\stackrel{\longrightarrow}{A\in
\mathscr{A}_k}}K_0(Var_{A})\rightarrow K_0(Var_{k})
\\ \psi&:&\lim_{\stackrel{\longrightarrow}{A\in
\mathscr{A}_k}}\mathcal{M}_A\rightarrow \mathcal{M}_{k}
\end{eqnarray*}
are isomorphisms.
\end{prop}
\begin{proof}
Surjectivity follows from the fact that for any $k$-variety $X$,
there exist an object $A$ in $\mathscr{A}_k$ and a $A$-model $X_A$
for $X$, by \cite[8.8.2]{ega4.3}. Injectivity follows from the
following facts:

 if $A$ is an object of $\mathscr{A}_k$,
and $U_A$ and $V_A$ are $A$-varieties, then  the canonical map
$$\lim_{\stackrel{\longrightarrow}{A\subset A'\in
\mathscr{A}_k}}Hom_{A'}(U_A\times_A A',V_A\times_A A')\rightarrow
Hom_{k}(U_A\times_A k,V_A\times_A k)$$ is a bijection
\cite[8.8.2]{ega4.3}. Moreover, if $f_A:U_A\rightarrow V_A$ is a
morphism of $A$-varieties such that the induced morphism
$f_k:U_A\times_A k\rightarrow V_A\times_A k$ is a closed (resp.
open) immersion, then there exists an object $A'$ in
$\mathscr{A}_k$ with $A\subset A'$ such that the natural morphism
$f_{A'}:U_A\times_A A'\rightarrow V_A\times_A A'$ is a closed
(resp. open) immersion \cite[8.10.5]{ega4.3}.
\end{proof}
Proposition \ref{spread} provides a convenient way to construct
additive and multiplicative invariants of $k$-varieties. We will
give an illustration in Section \ref{sec-poingen} (see also the
Appendix).
\subsection{The Poincar\'e polynomial}\label{sec-poingen}
Let $k$ be any field. It is, in general, a non-trivial problem to
decide whether the classes of two $k$-varieties $X,\,Y$ in
$K_0(Var_k)$ are distinct. To this aim, it is important to know
some ``computable'' realization morphisms on $K_0(Var_k)$. If $k$
has characteristic zero, we've encountered many of these in the
preceding sections, but in positive characteristic, we're
considerably less equiped. For $k=\C$ or $k$ a finite field, one
can define the \textit{virtual Betti numbers} $\beta_i(X)$ and the
\textit{Poincar\'e polynomial} $P(X;T)$ of a $k$-variety $X$ using
Deligne's theory of weights \cite{deligne-hodgeIII}
\cite{deligne-weilII}. By spreading out, these invariants can be
generalized to arbitrary base fields. These definitions  seem to
be known to experts, but since we could not find a reference for
their construction and main properties, we found it worthwhile to
include the arguments in the Appendix. We summarize in the
following theorem the facts we'll need in the remainder of this
article.
\begin{theorem}\label{theo-poin}
Let $k$ be any field. For any separated $k$-scheme of finite type
$X$, its Poincar\'e polynomial $P(X;T)\in \Z[T]$ has degree $2d$,
with $d$ the dimension of $X$. The coefficient of $T^{2d}$ in
$P(X;T)$ equals the number of irreducible components of dimension
$d$ of $X\times_k k^s$. The value $P(X;1)$ is equal to the Euler
characteristic $\chi_{top}(X)$. If $X$ is smooth and proper over
$k$, then
$$P(X;T)=\sum_{i\geq 0}(-1)^ib_i(X)T^i$$ with
$b_i(X)=\mathrm{dim}H^i(X\times_k k^s,\Q_\ell)$ for any prime
$\ell$ invertible in $k$.

There exists a unique ring morphism $P:K_0(Var_k)\rightarrow
\Z[T]$ mapping $[X]$ to $P(X;T)$ for any separated $k$-scheme of
finite type $X$. The morphism $P$ maps $\LL$ to $T^2$ and
localizes to a ring morphism $P:\mathcal{M}_k\rightarrow
\Z[T,T^{-1}]$.
\end{theorem}
\begin{proof}
See Appendix (Section \ref{Appendix}), in particular Propositions
\ref{poin-euler}, \ref{components} and \ref{smoothbetti}.
\end{proof}

The existence and properties of the Poincar\'e polynomial yield
the following useful criterion to distinguish elements of the
localized Grothendieck ring.
\begin{cor}\label{nonzero}
Let $k$ be any field, and let $X$ and $Y$ be separated $k$-schemes
of finite type such that $[X]=[Y]$ in $\mathcal{M}_k$. Then $X$
and $Y$ have the same dimension $d$, and $X\times_k k^s$ and
$Y\times_k k^s$ have the same number of irreducible components of
dimension $d$. In particular, if $X$ is non-empty, then $[X]\neq
0$ in $\mathcal{M}_k$. If $X$ and $Y$ are proper and smooth over
$k$, then they have the same $\ell$-adic Betti numbers (for $\ell$
invertible in $k$).
\end{cor}
The first part of Corollary \ref{nonzero} (concerning the
dimension and the geometric number of irreducible components of
maximal dimension) was proven in \cite[4.7]{lalu-julienliu} by a
different method (their proof was formulated for $K_0(Var_k)$ but
holds also for $\mathcal{M}_k$).
\subsection{Zero divisors}
In \cite{Poo}, Poonen has shown that $K_0(Var_k)$ is not a domain
if $k$ is a field of characteristic zero. Other examples of
zero-divisors were constructed by Koll\'ar \cite[Ex.\,6]{kollar}
and by Liu and Sebag \cite[5.11]{lalu-julienliu}.
 To my best understanding, these proofs don't say anything about
$\mathcal{M}_k$. The authors construct elements $\alpha$ and
$\beta$ in $K_0(Var_k)$ such that $\alpha\cdot \beta=0$, and to
show that neither $\alpha$ nor $\beta$ are zero, they use the
stably birational realization $\Phi_{SB}$  (Theorem \ref{stabir})
or the Albanese realization (Corollary \ref{albanese}). However,
each of these realization morphisms maps $\LL$ to $0$, so they do
not allow to conclude that $\alpha$ and $\beta$ are non-zero in
$\mathcal{M}_k$.

To my knowledge, the only case where it has been shown that
$\mathcal{M}_k$ is not a domain, is the case where $k$ is not
separably closed \cite[3.5]{rokaeus} (the result is stated there
for $K_0(Var_k)$ but works also for $\mathcal{M}_k$; it
generalizes \cite[Thm.\,25]{naumann}). We'll give a new proof of
this result, which does not use $\ell$-adic cohomology. We refer
to Proposition \ref{not-domain} for another example of a
zero-divisor in $\mathcal{M}_k$ (for appropriate $k$).

\begin{prop}\label{not-domain2}
If $k$ is any field which is not separably closed, then
$K_0(Var_k)$ and $\mathcal{M}_k$ are not domains.
\end{prop}
\begin{proof}
Choose a non-trivial finite Galois extension $k'$ of $k$, and put
$d=[k':k]$. Then $([\Spec k']-d)\cdot [\Spec k']=0$ in
$K_0(Var_k)$.  We'll prove that $[\Spec k']\neq 0$ and $[\Spec
k']\neq d$ in $\mathcal{M}_k$.
%
%
 By Proposition \ref{spread}, it is enough to show that $[X]\neq 0,d$
in $\mathcal{M}_A$ for every object $A$ in $\mathscr{A}_k$ and
every $A$-model $X$ of $\Spec k'$. If $y$ is a closed point on $X$
then $k(y)$ is finite, and applying the point counting morphism
$\sharp:\mathcal{M}_y\rightarrow \Q$ (Section \ref{sec-real}) we
see that $[X\times_{\Spec A}y]\neq 0$ in $\mathcal{M}_y$. This
implies $[X]\neq 0$ in $\mathcal{M}_A$. It remains to show that
$[X]\neq d$ in $\mathcal{M}_A$.

Localizing $A$ we may assume that $X$ is irreducible. The function
field $k(X)$ is a field extension of degree $d$ of the quotient
field $k(A)$ of $A$, since $k(X)\otimes_{k(A)}k\cong k'$. Base
changing to an object $A'$ in $\mathscr{A}_k$ with $A\subset A'$
we may assume that every automorphism of $k'$ over $k$ is induced
by an automorphism of $X$ over $A$ \cite[8.8.2]{ega4.3}.
Localizing $A$ again we may suppose that $X$ is a Galois cover of
$\Spec A$.

By Chebotarev's density theorem for $\Spec A$ (see \cite{Serre2}),
there exists a closed point $x$ on $\Spec A$ which does not split
completely in the Galois cover $X$. It suffices to show that
$[X\times_A k(x)]\neq d$ in $\mathcal{M}_{x}$. This can be seen by
applying the point counting morphism
$\sharp:\mathcal{M}_x\rightarrow \Q$.
\end{proof}

I do not know whether $\mathcal{M}_k$ is a domain if $k$ is
separably closed, or if $K_0(Var_k)$ is a domain if $k$ is
separably closed and has characteristic $p>0$. If $k'$ is a purely
inseparable finite field extension of $k$, I do not know if
$[\Spec k']\neq [\Spec k]$ in $K_0(Var_k)$.

\section{N\'eron smoothening of pairs}\label{sec-smooth}
\subsection{Pairs of varieties}
Let $S$ be any scheme. A pair of $S$-varieties $(X,A)$ consists of
a $S$-variety $X$ and a closed subvariety $A$ of $X$. We say that
the pair $(X,A)$ is proper, smooth,\,$\ldots$ if this holds for
both $X$ and $A$. A morphism of pairs of $S$-varieties
$f:(Y,B)\rightarrow (X,A)$ is a morphism of $S$-varieties
$f:Y\rightarrow X$ such that $f(B)\subset A$. Since $B$ is
reduced, this implies that the restriction of $f$ to $B$ factors
through a morphism of $S$-varieties $f:B\rightarrow A$. We embed
the category of $S$-varieties in the category of pairs by
$X\mapsto (X,\emptyset)$.

We denote by $R$ a discrete valuation ring, with quotient field
$K$ and residue field $k$.
The maximal ideal of $R$ will be denoted by
$\mathfrak{M}$. We fix a separable closure $K^s$ of $K$, and we
denote by $R^{sh}$ the strict henselization of $R$ in $K^s$, and
by $K^{sh}\subset K^s$ its quotient field.
We denote by $k^s$
the residue field of $R^{sh}$. The field $k^s$ is a separable
closure of $k$.

 If $(X,A)$ is a pair of $R$-varieties,
then their generic fibers $(X_K,A_K)$ form a pair of
$K$-varieties. We say that the pair $(X,A)$ is generically smooth
if $(X_K,A_K)$ is a smooth pair of $K$-varieties.

We recall two properties which we'll frequently use: if $Y$ is a
smooth $k$-variety, then $Y(k^s)$ is schematically dense in $Y$
\cite[2.2.13]{neron} and if $X$ is a smooth $R$-variety, then the
natural reduction map $X(R^{sh})\rightarrow X_s(k^s)$ is
surjective \cite[2.3.5]{neron}.
\subsection{N\'eron smoothening}
\begin{definition}[N\'eron smoothening]\label{def-nersmooth}
If $X$ is a generically smooth $R$-variety, then a N\'eron
smoothening of $X$ is a morphism of $R$-varieties $h:Y\rightarrow
X$ with the following properties:
\begin{itemize}
\item $Y$ is smooth over $R$ \item $h_K:Y_K\rightarrow X_K$ is an
isomorphism \item $h$ satisfies the following ``weak valuative
criterion'': the natural map $\phi:Y(R^{sh})\rightarrow X(R^{sh})$
is bijective.
\end{itemize}
\end{definition}
Note that injectivity of $\phi$ follows already from the fact that
$h$ is separated. Any generically smooth $R$-variety $X$ admits a
N\'eron smoothening, by \cite[3.5.2]{neron}.

Definition \ref{def-nersmooth} is different from the one in
\cite[3.1.1]{neron} but more adapted to our purposes. To compare
both definitions, let us call a morphism of $R$-varieties
$h:Y\rightarrow X$ a smoothening$*$ if it is a smoothening in the
sense of \cite[3.1.1]{neron} (i.e. $X_K$ is smooth over $K$, $h_K$
is an isomorphism, $h$ is proper, and the natural map
$Sm(Y)(R^{sh})\rightarrow Y(R^{sh})$ is bijective).

\begin{definition}[Admissible ideal sheaf]
If $Y$ is any $R$-variety, an ideal sheaf $\mathcal{I}$ on $Y$ is
called admissible if it contains an element of the maximal ideal
$\mathfrak{M}$ of $R$.
\end{definition}
\begin{lemma}\label{nosmooth}
Let $Y$ be any $R$-variety, let $\mathcal{I}$ be an admissible
locally principal ideal sheaf on $Y$, and denote by $Z$ the closed
subscheme of $Y$ defined by $\mathcal{I}$. If
$$\{a\in Y(R^{sh})\,|\,a_s\in Z(k^s)\}=\emptyset$$
then $Sm(Y)\subset Y-Z$.
\end{lemma}
\begin{proof}
We may as well assume that $Y$ is connected and smooth. Then
$Z_{red}$ either is empty or coincides with the special fiber
$Y_s$. But $Y_s(k^s)$ is dense in $Y$, and any point in $Y_s(k^s)$
lifts to a section in $Y(R^{sh})$, so $Z$ is empty.
\end{proof}
\begin{prop}
Let $X$ be a generically smooth $R$-variety. If $Z\rightarrow X$
is a smoothening$*$, then the induced morphism $Sm(Z)\rightarrow
X$ is a N\'eron smoothening in the sense of Definition
\ref{def-nersmooth}. Conversely, if $h:Y\rightarrow X$ is a
N\'eron smoothening, then there exists a smoothenening$*$
$g:Z\rightarrow X$ such that $Y$ and $Sm(Z)$ are isomorphic as
$X$-schemes.
\end{prop}
\begin{proof}
It is obvious from the definition that $Sm(Z)\rightarrow X$ is a
N\'eron smoothening if $Z\rightarrow X$ is a smoothening$*$.
Conversely, let $h:Y\rightarrow X$ be a smoothening$*$. By
Nagata's embedding theorem, there exist a proper morphism
$\overline{h}:\overline{Y}\rightarrow X$ and a dense open
immersion $j:Y\rightarrow \overline{Y}$ such that
$h=\overline{h}\circ j$. Since $h_K$ is an isomorphism,
$\overline{h}_K$ and $j_K$ are isomorphisms, and since $Y\subset
Sm(\overline{Y})$ and $h$ is a N\'eron smoothening, $\overline{h}$
is a smoothening$*$.

If $k$ is perfect, this implies automatically that $j$ is an
isomorphism onto $Sm(\overline{Y})$ (because any point in
$Sm(\overline{Y})_s(k^s)$ lifts to a section in
$\overline{Y}(R^{sh})$, which has to be contained in $Y(R^{sh})$
since $Y\rightarrow X$ is a N\'eron smoothening).

If $k$ is not perfect,
 this needs not be true (take $X=\overline{Y}=\A^1_R$ and $Y=X-\{x\}$
with $x$ any closed point on $X_s$ whose residue field is
inseparable over $k$) so we have to modify $\overline{Y}$. Let $U$
be the complement of $Y_s$ in $\overline{Y}_s$, with its reduced
closed subscheme structure, and denote by $\mathcal{I}$ its
defining ideal sheaf. Let $b:Z\rightarrow \overline{Y}$ be the
blow-up with center $U$. Since $b$ is an isomorphism over $Y$,
$$\overline{h}\circ b:Z\rightarrow X$$ is again a smoothening$*$.

We put
$$V:=Z\times_{\overline{Y}}U$$
This is a closed subscheme of $Z$, defined by the invertible sheaf
$\mathcal{I}\mathcal{O}_Z$. Since $Y\rightarrow X$ is a N\'eron
smoothening,
$$\{a\in Z(R^{sh})\,|\,a\in V(k^s)\}=\emptyset$$
so Lemma \ref{nosmooth} implies that $V$ is disjoint from $Sm(Z)$.
On the other hand, $b:Z-V\rightarrow Y$ is an isomorphism and $Y$
is smooth, so we conclude that $Sm(Z)=Z-V$, and $Y$ and $Sm(Z)$
are isomorphic as $X$-schemes.
\end{proof}

\begin{definition}[N\'eron smoothening of pairs]\label{def-nerpair}
Let $(X,A)$ be a generically smooth pair of $R$-varieties. A
N\'eron smoothening of $(X,A)$ is a morphism of pairs of
$R$-varieties $h:(Y,B)\rightarrow (X,A)$ such that $h:Y\rightarrow
X$ is a N\'eron smoothening of $X$ and $h:B\rightarrow A$ is a
N\'eron smoothening of $A$.
\end{definition}
This definition implies in particular that $B$ is the schematic
closure of $h_K^{-1}(A_K)\subset Y_K$ in $Y$. Note that
$h:(Y,B)\rightarrow (X,A)$ is a N\'eron smoothening as soon as
$h:Y\rightarrow X$ is a N\'eron smoothening, $B$ is smooth over
$R$, and $h_K:B_K\rightarrow A_K$ is an isomorphism: then a
section $a$ in $A(R^{sh})$ lifts uniquely to a section $a'$ in
$Y(R^{sh})$, which is automatically included in $B(R^{sh})$ since
$a'_K\in B_K(K^{sh})$ and $B$ is closed in $Y$.

\begin{definition}[Strict transform and admissible blow-up]
Let $(X,A)$ be a pair of $R$-varieties. If $h:Y\rightarrow X$ is a
morphism of $R$-varieties such that $h_K$ is an isomorphism, then
the strict transform of $A$ in $Y$ is the schematic closure of
$h_K^{-1}(A_K)$ in $Y$.

If $\mathcal{I}$ is an admissible ideal sheaf on $X$, we define
the blow-up of $(X,A)$ at the center $\mathcal{I}$ as the morphism
of pairs of $R$-varieties
$$h:(Y,B)\rightarrow (X,A)$$
where $h:Y\rightarrow X$ is the blow-up of $X$ at $\mathcal{I}$,
and $B$ is the strict transform of $A$ in $Y$. We call such a
morphism $h$ an admissible blow-up of $(X,A)$.
\end{definition}
We denote for any $R$-variety $Z$ by $Z^{flat}$ the maximal
$R$-flat closed subscheme of $Z$, i.e. the closed subscheme of $Z$
defined by the $\mathfrak{M}$-torsion ideal. Then the admissible
blow-up of $(X,A)$ at the ideal $\mathfrak{M}\mathcal{O}_X$ is the
natural morphism $(X^{flat},A^{flat})\rightarrow (X,A)$.

In general, if $(Y,B)\rightarrow (X,Z)$ is any admissible blow-up,
then $B$ is flat over $R$. Moreover, the natural maps
$Y(R^{sh})\rightarrow X(R^{sh})$ and $B(R^{sh})\rightarrow
A(R^{sh})$ are bijections, by the valuative criterion for
properness; so we can identify any subset $E$ of $X(R^{sh})$
(resp. $A(R^{sh})$) with its inverse image in $Y(R^{sh})$ (resp.
$B(R^{sh})$).


The main result of this section is the following.
\begin{theorem}\label{smoothpair}
Let $(X,A)$ be a generically smooth pair of $R$-varieties. There
exists a composition $h:(Y,B)\rightarrow (X,A)$ of admissible
blow-ups, such that $Sm(B)=Sm(Y)\cap B$ and such that the
restriction of $h$ to $(Sm(Y),Sm(B))\rightarrow (X,A)$ is a
N\'eron smoothening of $(X,A)$.
%
\end{theorem}
The proof of Theorem \ref{smoothpair} follows after Proposition
\ref{closedsub}. First, we need some preliminary results.
\begin{lemma}\label{smoothlift}
Let $Z$ be a $R$-variety, and let $a$ be a section in $Z(R^{sh})$.
Let $C$ be a closed subscheme of $Z_s$, and assume that $Z$ is
smooth over $R$ at $a_s\in Z_s(k^s)$, and that $C$ is smooth over
$k$ at $a_s$. Denote by $Z'\rightarrow Z$ the blow-up with center
$C$, and by $a'$ the unique lifting of $a$ to $Z'(R^{sh})$. Then
$Z'$ is smooth over $R$ at $a'_s\in Z'_s(k^s)$.
\end{lemma}
\begin{proof}
Since blowing up commutes with flat base change, we may assume
that $Z$ is smooth over $R$, and that $C$ is smooth over $k$.
Denote by $D\rightarrow Z$ the dilatation with center $C$ (see
\cite[\S\,3.2]{neron}); then $D$ is an open subscheme of $Z'$ in a
natural way. By the universal property of the dilatation
\cite[3.2.1]{neron}, $a'$ factors through $D$, so $a'_s\in
D_s(k^s)$. But $D$ is smooth over $R$ by \cite[3.2.3]{neron}.
\end{proof}
\begin{prop}\label{smoothen}
Let $(X,A)$ be a pair of generically smooth $R$-varieties, and
assume that the natural map $Sm(X)(R^{sh})\rightarrow X(R^{sh})$
is bijective. 
There exists a
composition $h:(Y,B)\rightarrow (X,A)$ of admissible blow-ups
such that $Sm(B)\rightarrow B$ and $Sm(Y)\rightarrow Y$ are
N\'eron smoothenings.
\end{prop}
\begin{proof}
We may assume that $A$ is flat over $R$. By \cite[3.4.2]{neron},
there exists a composition
$$\begin{CD}A'=A_r@>h_{r-1}>>\ldots@>h_1>>A_1@>h_0>> A_0=A\end{CD}$$ such that
$h_i$ is the blow-up at a closed subscheme $C_i$ of the special
fiber $(A_i)_s$, for $i=0,\ldots,r-1$, and such that the natural
map $Sm(A')(R^{sh})\rightarrow A'(R^{sh})=A(R^{sh})$ is bijective.
Moreover, we may assume that each center $C_i$ is $E$-permissible
(in the sense of \cite[p.\,71]{neron}) with $E=A(R^{sh})$. This
implies in particular that the $k$-smooth locus $U_i$ of $C_i$ is
open and dense in $C_i$, and that none of the $k^s$-valued points
of $C_i-U_i$ lift to a section in $A_i(R^{sh})$.

Now let $X'\rightarrow X$ be the composition
$$\begin{CD}X'=X_r@>g_{r-1}>>\ldots@>g_1>>X_1@>g_0>>
X_0=X\end{CD}$$ with $g_i$ the blow-up of $X_i$ at $C_i$, for
$i=0,\ldots,r-1$. 
Then $A_i$ is canonically
isomorphic to the strict transform of $A$ in $X_i$, for each $i$,
and these isomorphisms identify the restriction of $g_i$ to $A_i$
with the morphism $h_i$. In particular, $A'$ is canonically
isomorphic to the strict transform of $A$ in $X'$.

If $a$ is any section in $A'(R^{sh})$, then $X'$ is smooth over
$R$ at $a_s\in X'_s(k^s)$, by Lemma \ref{smoothlift}. This implies
that $Sm(A')\subset Sm(X')$, because every point in
$Sm(A')_s(k^s)$ lifts to a section in $Sm(A')(R^{sh})$ and
$Sm(A')_s(k^s)$ is schematically dense in $Sm(A')_s$. Again
applying \cite[3.4.2]{neron}, we can find a composition
$f:Y\rightarrow X'$ of admissible blow-ups such that $f$ is an
isomorphism over $Sm(X')$, and such that
$Sm(Y)(R^{sh})=Y(R^{sh})$. If we denote by $B$ the strict
transform of $A'$ in $Y$, then the map $B\rightarrow A'$ is an
isomorphism over $Sm(A')$, and in particular,
$Sm(B)(R^{sh})=B(R^{sh})$.
\end{proof}

The result in Proposition \ref{smoothen} is not yet strong enough
to produce a N\'eron smoothening of the pair $(X,A)$, since it
does not guarantee that $Sm(B)$ is a closed subscheme of $Sm(Y)$.
For this purpose, we introduce a new invariant.
\begin{definition}
Let $(X,A)$ be a pair of $R$-varieties, and denote by
$\mathcal{I}_A$ the defining ideal sheaf of $A$ on $X$. If $a$ is
a section in $X(R^{sh})$, and $x$ is the image of $a_s$ in $X$,
then we define the contact of $a$ and $A$ by
$$c_A(a)=\{\min v(a^*f)\,|\,f\in (\mathcal{I}_{A})_x\}\in \N\cup\{\infty\}$$
where $v$ denotes the discrete valuation on $R^{sh}$.
\end{definition}
Note that $c_A(a)=0$ iff $x\notin A$, and $c_A(a)=\infty$ iff
$a\in A(R^{sh})$.

\begin{lemma}\label{bound}
Assume that $R$ is excellent. Let $(X,A)$ be a generically smooth
pair of $R$-varieties, and let $C$ be a closed subscheme of $A_s$.
Put
$$E_C=\{a\in X(R^{sh})\,|\,a_s\in C(k^s)\}$$
and assume that $E_C\cap A(R^{sh})=\emptyset$. Then there exists a
value $c>0$ such that $c_A(a)\leq c$ for every $a\in E_C$.
\end{lemma}
\begin{proof}
By \cite[5.6]{greco} $R^{sh}$ is excellent, so we may assume that
$R=R^{sh}$, and that there exists a closed immersion $X\rightarrow
\A^n_R$ for some $n>0$. Let $F_1,\ldots,F_r$ be a system of
generators of the defining ideal of $A$ in $\A^n_R$, and assume
that $c_A$ is unbounded on $E_C$. Then in particular, for any
$\nu>0$, there exists a point $x\in \A^n_R(R)=R^n$ such that
$x_s\in C(k)$ and $F_i(x)\equiv 0\mod t^{\nu}$ for all $i$. Since
$R$ is excellent, it follows from Greenberg's Theorem
\cite[Thm.\,1]{Gr} that there exists a section $a\in E_C$ which is
contained in $A(R)$; so we arrive at a contradiction.
\end{proof}
\begin{lemma}\label{drop}
Let $(X,A)$ be a pair of $R$-varieties, let $a$ be a section of
$X(R^{sh})$ which is not contained in $A(R^{sh})$, and let $C$ be
a closed subscheme of $X_s$. Denote by $(X',A')\rightarrow (X,A)$
the admissible blow-up with center $C$. Then $c_{A'}(a)\leq
c_{A}(a)$. If, moreover, $C$ is a closed subscheme of $A_s$ and
$a_s\in C(k^s)$, then $c_{A'}(a)<c_{A}(a)$.
\end{lemma}
\begin{proof}
We may assume that $X$ is affine. We choose a uniformizer $\pi$ in
$R$. Let $f$ be an element of the defining ideal of $A$ in $X$
such that $c_{A}(a)=v(a^*f)$. Since the pull-back of $f$ to $X'$
vanishes on $A'$, we see immediately that $c_{A'}(a)\leq
c_{A}(a)$.

Now assume that $a_s\in C(k^s)$ and that $C$ is a closed subscheme
of $A_s$, and denote by $I_C$ the defining ideal of $C$ in $X$. If
we denote by $D\rightarrow X$ the dilatation of $X$ with center
$C$, then $D$ is an open subscheme of $X'$ and $a$ is contained in
$D(R^{sh})\subset X'(R^{sh})$. Moreover, since $f$ vanishes on $C$
and $\pi$ generates $I_C\mathcal{O}_{D}$, there exists an element
$f'$ in $\mathcal{O}_{X'}(D)$ such that $f'=\pi.f$. Then $f'$
vanishes on $A'_K$ because $f$ vanishes on $A_K$, and as $A'$ is
the schematic closure of $A'_K$ in $A$, we see that $f'$ vanishes
on $A'$. Moreover, $v(a^*f')=v(a^*f)-1$ so $c_{A'}(a)<c_{A}(a)$.
\end{proof}
\begin{prop}\label{closedsub}
Let $(X,A)$ be a pair of generically smooth $R$-varieties, and
assume that the natural morphisms $Sm(X)\rightarrow X$ and
$Sm(A)\rightarrow A$ are N\'eron smoothenings. There exist
a composition $(Y,B)\rightarrow (X,A)$ of admissible blow-ups with
centers contained in (the strict transform of) $A$, such that
$Sm(B)=Sm(Y)\cap B$, and such that $Sm(Y)\rightarrow Y$ and
$Sm(B)\rightarrow B$ are N\'eron smoothenings.
\end{prop}
\begin{proof}
Denote by $D$ the complement of $Sm(X)\cap Sm(A)$ in $Sm(A)$ (with
its reduced closed subscheme structure), and denote by
$\overline{D}$ its schematic closure in $A_s$. Since every point
of $Sm(A)_s(k^s)$ lifts to a section of $A(R^{sh})$, we see that
$Sm(A)_s(k^s)\subset Sm(X)_s(k^s)$, so $D(k^s)=\emptyset$. Denote
by $C$ the complement of $Sm(X)\cap Sm(A)$ in $Sm(X)\cap A$ (with
its reduced closed subscheme structure), and by $\overline{C}$ its
schematic closure in $A_s$, and put
$$E_{\overline{C}}=\{a\in X(R^{sh})\,|\,a_s\in \overline{C}(k^s)\}$$

Denote by $S$ the completion of $R^{sh}$. The morphism
$Sm(A)\times_R S\rightarrow A\times_R S$ is a N\'eron smoothening
by \cite[3.6.6]{neron}, so $Sm(A)(S)=A(S)$. Since $\overline{C}$
is disjoint from $Sm(A)$, we have
$$\{b\in A(S)\,|\,b_s\in \overline{C}(k^s)\}=\emptyset$$ Since $S$ is excellent, we
can apply Lemma \ref{bound}, and we see that
$$M(X,A):=\max \{c_A(a)\,|\,a\in E_{\overline{C}}\}$$ is well-defined and
finite (we put $\max \emptyset=0$). We will argue by induction on
$M(X,A)$.

\textit{Induction basis: assume $M(X,A)=0$.} This is only possible
if $E_{\overline{C}}=\emptyset$, since for every section $a\in
E_{\overline{C}}$, $a_s$ belongs to $A_s(k^s)$, so $c_{A}(a)>0$.
Moreover, since any point of $Sm(X)_s(k^s)$ lifts to a section in
$Sm(X)(R^{sh})$, $E_{\overline{C}}=\emptyset$ implies
$C(k^s)=\emptyset$.

If $k$ is perfect, then we get $C=\emptyset$, so $Sm(X)\cap
Sm(A)=Sm(X)\cap A$. Also, in this case $D(k^s)=\emptyset$ implies
that $D=\emptyset$ and $Sm(A)\subset Sm(X)$, so we are done.

If $k$ is not perfect, we consider the admissible blow-up
$$h:(Y,B)\rightarrow (X,A)$$ with center $\overline{C}\cup \overline{D}$.
 Then Lemma \ref{nosmooth} shows that $h^{-1}(\overline{C}\cup\overline{D})$ is disjoint from $Sm(Y)$
 and $Sm(B)$, so $h$ induces isomorphisms
$$ \begin{array}{lllllll}
 Sm(Y)&\cong& Sm(X)-(\overline{C}\cup \overline{D})&=&Sm(X)-C& &
\\ Sm(B)&\cong& Sm(A)-(\overline{C}\cup \overline{D})&=&
Sm(A)-D&=&Sm(A)\cap Sm(X)
\end{array}$$
 Hence,
 $Sm(B)=Sm(Y)\cap B$, and $Sm(Y)\rightarrow
Y$ and $Sm(B)\rightarrow B$ are N\'eron smoothenings.

\textit{Induction step: assume $M:=M(X,A)>0$, and suppose that
Proposition \ref{closedsub} holds for all pairs as in the
statement with $M(\cdot,\cdot)<M$.} Let $h_1:(X_1,A_1)\rightarrow
(X,A)$ be the admissible blow-up with center $\overline{C}$. By
\cite[3.4.2]{neron} there exists a composition of admissible
blow-ups $h_2:(X_2,A_2)\rightarrow (X_1,A_1)$ such that
$Sm(X_2)\rightarrow X_2$ is a N\'eron smoothening. Applying
Proposition \ref{smoothen}, we may suppose that
$Sm(A_2)\rightarrow A_2$ is also a N\'eron
smoothening.

Denote by $C_2$ the complement of $Sm(X_2)\cap Sm(A_2)$ in
$Sm(X_2)\cap A_2$ (with its reduced closed subscheme structure),
and by $\overline{C}_2$ its schematic closure in $(A_2)_s$. We put
$$E_{\overline{C}_2}=\{a\in X_2(R^{sh})\,|\,a_s\in \overline{C}_2(k^s)\}$$
Since $\overline{C}_2\subset (h_1\circ h_2)^{-1}(\overline{C})$,
Lemma \ref{drop} implies that $c_{A_2}(a)<M$ for each element $a$
of $E_{\overline{C}_2}$, so $M(X_2,A_2)<M$ and we may conclude by
the induction hypothesis.
%
\end{proof}

\begin{proof}[Proof of Theorem \ref{smoothpair}]
By \cite[3.4.2]{neron}, there exists a composition of admissible
blow-ups $(X',A')\rightarrow (X,A)$ such that the natural map
$Sm(X')(R^{sh})\rightarrow X'(R^{sh})$ is a bijection. By
Proposition \ref{smoothen} we can find a composition
$h:(X'',A'')\rightarrow (X',A')$ of admissible blow-ups such that
the maps $Sm(A'')\rightarrow A''$ and $Sm(X'')\rightarrow X''$ are
N\'eron smoothenings. Finally, we apply Proposition
\ref{closedsub} to the pair $(X'',A'')$.
%
\end{proof}
\begin{definition}[Weak N\'eron models of
pairs]\label{def-weakner}
If $(X_K,A_K)$ is a smooth pair of $K$-varieties, then a weak
N\'eron model for $(X_K,A_K)$ is a smooth pair of $R$-varieties
$(Y,B)$ endowed with an isomorphism of pairs of $K$-varieties
$f:(Y_K,B_K)\rightarrow (X_K,A_K)$ such that the natural map
$Y(R^{sh})\rightarrow Y_K(K^{sh})$ is a bijection.
\end{definition}
Note that $B(R^{sh})\rightarrow B_K(K^{sh})$ will automatically be
a bijection: any section $a$ in $Y(R^{sh})$ with $a_K\in
B_K(K^{sh})$ belongs to $B(R^{sh})$, since $B$ is closed in $Y$.

If $V_K$ is a smooth $K$-variety, then a smooth $R$-variety $W$
endowed with an isomorphism of $K$-varieties $g:W_K\rightarrow
V_K$ is a weak N\'eron model for $V_K$ (in the sense of
\cite{neron}) iff $(W,\emptyset)$ is a weak N\'eron model for
$(V_K,\emptyset)$ w.r.t. the map $g$. Moreover, $((Y,B),f)$ is a
weak N\'eron model for $(X,A)$ iff $(Y,f)$ is a weak N\'eron model
for $X$ and $(B,f|_{B_K})$ is a weak N\'eron model for $A$.

The following proposition gives a necessary and sufficient
condition for the existence of a weak N\'eron model.
\begin{prop}\label{iff}
A smooth pair of $K$-varieties $(X_K,A_K)$ admits a weak N\'eron
model, iff $X_K(K^{sh})$ is bounded in $X_K$ (in the sense of
\cite[1.1.2]{neron}).
\end{prop}
\begin{proof}
By \cite[3.5.7]{neron}, $X_K(K^{sh})$ is bounded in $X_K$ iff
there exists a $R$-variety $X'$ whose generic fiber is isomorphic
to $X_K$ and such that the natural map $X'(R^{sh})\rightarrow
X'_K(K^{sh})$ is a bijection, so boundedness is obviously a
necessary condition for the existence of a weak N\'eron model.

Let us prove that it is also sufficient. We fix an isomorphism
between $X'_K$ and $X_K$. If we denote by $A'$ the
scheme-theoretic closure of $A_K$ in $X'$, then the generic fiber
of $(X',A')$ is isomorphic to the pair $(X_K,A_K)$.

By Theorem \ref{smoothpair}, there exists a N\'eron smoothening
$g:(Y,B)\rightarrow (X',A')$. The pair $(Y,B)$ endowed with the
isomorphism $g_K:(Y_K,B_K)\rightarrow (X'_K,A'_K)\cong (X_K,A_K)$
is a weak N\'eron model of $(X_K,A_K)$.
\end{proof}
\begin{cor}
Any smooth and proper pair of $K$-varieties admits a weak N\'eron
model.
\end{cor}
\begin{proof}
For any proper $K$-variety $X_K$, $X_K(K^{sh})$ is bounded in
$X_K$ by \cite[1.1.6]{neron}.
\end{proof}

 We'll take a closer look at this
boundedness condition in the next section.

\section{Bounded varieties and weak N\'eron
models}\label{sec-bound} We keep the notations of Section
\ref{sec-smooth}, and we assume moreover that $R$ is complete.

\begin{definition}[Bounded and smoothly bounded
varieties]\label{def-bounded-alg} Let $L$ be a discretely valued
field, and let $X$ be a $L$-variety. We say that $X$ is bounded if
$X(L^{sh})$ is bounded in $X$ (in the sense of
\cite[1.1.2]{neron}). We say that $X$ is smoothly bounded if $X$
is bounded and the natural map $Sm(X)(L^{sh})\rightarrow
X(L^{sh})$ is a bijection.
\end{definition}
\begin{remark}
If $Sm(X)$ is bounded and the natural map
$Sm(X)(L^{sh})\rightarrow X(L^{sh})$ is a bijection, then $X$ is
smoothly bounded by \cite[1.1.4]{neron}. The converse holds if the
ring of integers of $L$ is excellent \cite[1.1.9]{neron}.
\end{remark}
\begin{definition}[Bounded and smoothly bounded rigid varieties]\label{def-bounded}
We say that a rigid $K$-variety $X$ is bounded if it is separated
and there exists a quasi-compact open subvariety $V$ of $X$ such
that the natural map $V(K')\rightarrow X(K')$ is a bijection for
each finite unramified extension $K'$ of $K$. If, moreover, $V$ is
smooth, then we call $X$ smoothly bounded.
\end{definition}

If $X$ is a rigid $K$-variety, then strictly speaking, the set
$X(K^a)$ is not defined since $K^a$ is not an affinoid
$K$-algebra. Therefore, we put $X(K^a)=\cup_{K'/K}X(K')$ where
$K'$ runs trhough the finite extensions of $K$ inside the fixed
algebraic closure $K^a$. The set $X(K^{sh})$ is defined
similarily. If $Y$ is a $K$-variety, then the analytification map
$Y^{an}\rightarrow Y$ induces natural bijections $Y(K^a)=
Y^{an}(K^a)$ and $Y(K^{sh})=Y^{an}(K^{sh})$.

The definition of a bounded rigid variety appeared earlier in
\cite[1.2]{formner} and \cite[5.6]{NiSe-weilres}. The following
proposition compares it to Definition \ref{def-bounded-alg} for
algebraic varieties.

\begin{prop}\label{comparbounded}
Let $X$ be an algebraic variety over $K$, and denote by $X^{an}$
its analytification.

$(a)$ If $E$ is a subset of $X(K^{a})$, then $E$ is bounded in $X$
iff there exists a quasi-compact open subvariety $V$ of $X^{an}$
such that $E$ is contained in $V(K^a)$.

$(b)$ In particular, $X^{an}$ is bounded iff $X$ is bounded, and
$X^{an}$ is smoothly bounded iff $X$ is smoothly bounded.
\end{prop}
\begin{proof}
If $E$ is bounded in $X$, the existence of a subvariety $V$ as in
the statement follows easily from the definition
\cite[1.1.2]{neron}. So suppose conversely that $V$ is a
quasi-compact open subvariety of $X^{an}$ such that $E$ is
contained in $V(K^a)$.

Choose a finite cover of $X$ by affine open subchemes
$U_1,\ldots,U_r$. It is clear from the definition
\cite[1.1.2]{neron} and the Maximum Modulus Principle
\cite[6.2.1.4]{BGR} that, for any affine $K$-variety $U$, a subset
$F$ of $U(K^a)$ is bounded in $U$ iff there exists a quasi-compact
open subvariety $W$ of $U^{an}$ such that $E$ is contained in
$W(K^{a})$. Therefore, it suffices to construct, for each index
$i\in \{1,\ldots,r\}$, a quasi-compact open subvariety $V_i$ of
$(U_i)^{an}$ such that $V(K^a)\subset\cup_{i}V_i(K^a)$.

Now $\{(U_1)^{an},\ldots,(U_r)^{an}\}$ is an admissible open cover
of $X^{an}$ \cite[0.3.3]{bert}, and $$\{(U_1)^{an}\cap
V,\ldots,(U_r)^{an}\cap V\}$$ is an admissible open cover of $V$.
Since $V$ is quasi-compact, this cover can be refined by a finite
affinoid cover $\mathcal{W}=\{W_1,\ldots,W_q\}$. If we define
$V_i$ as the union of those members $W_j$ of the cover
$\mathcal{W}$ which are contained in $(U_i)^{an}$, for
$i=1,\ldots,r$, then $V_i$ is a quasi-compact open subvariety of
$(U_i)^{an}$ and $\cup_{i}V_i(K^a)=V(K^a)$. This concludes the
proof of $(a)$.

Applying this result to $E=X(K^{sh})$ we see that $X^{an}$ is
bounded iff $X$ is bounded. Since, moreover, $X$ is smooth at a
closed point $x$ iff $X^{an}$ is smooth at $x$
\cite[5.2.1]{conrad}, we see that $X^{an}$ is smoothly bounded iff
$X$ is smoothly bounded.
\end{proof}

\begin{cor}\label{corcompar}
Let $S$ be a discrete valuation ring, with quotient field $L$ and
residue field $k$, and let $R$ be its completion. Let $X$ be a
 $L$-variety, and assume either that $X$ is smooth or that $S$ is excellent.
  Then $(X\times_L K)^{an}$ is bounded iff $X$ is bounded, and
 $(X\times_L K)^{an}$ is smoothly bounded
iff $X$ is smoothly bounded.
\end{cor}
\begin{proof}
We fix an embedding of $L^s$ in $K^s$. We know from Proposition
\ref{comparbounded} that $(X\times_L K)^{an}$ is bounded iff
$(X\times_L K)(K^{sh})$ is bounded in $X\times_L K$. This is also
equivalent to the property that $X(K^{sh})$ is bounded in $X$
\cite[1.1.5]{neron}, which implies that $X(L^{sh})$ is bounded in
$X$.

Assume, conversely, that $X$ is bounded. We have to show that
$X(K^{sh})$ is bounded in $X$. If we denote by $K'$ the closure of
$L^{sh}$ inside the completion of $K^{sh}$, then $K^{sh}$ is a
subfield of $K'$. By \cite[1.1.5]{neron}, $X(L^{sh})$ (viewed as a
subset of $X(K')$) is bounded in $X\times_L K'$. Since $X$ is
smooth or $S$ (and hence $S^{sh}$ \cite[5.6]{greco}) is excellent,
we can apply \cite[3.6.10]{neron} and we see that $X(L^{sh})$ is
dense in $X(K')$ (w.r.t. the topology induced by the valuation on
$L$). It is clear from the definition \cite[1.1.2]{neron} that
this implies that $X(K')$ is bounded in $X$. Therefore,
$X(K^{sh})$ is bounded in $X$.

Now assume that $S$ is excellent. By \cite[17.7.2]{ega4.4},
$Sm(X\times_L K)$ is canonically isomorphic to $Sm(X)\times_L K$.
Combining this with Proposition \ref{comparbounded}, we see that
$X$ is smoothly bounded if $(X\times_L K)^{an}$ is smoothly
bounded. Conversely, if $X$ is smoothly bounded, then $Sm(X)$ is
bounded since $S$ is excellent, so
 $(Sm(X)\times_L
K)^{an}$ is bounded by the first part of Corollary
\ref{corcompar}. Hence, to show that $(X\times_L K)^{an}$ is
smoothly bounded, it suffices to show that the natural map
$$Sm((X\times_L K)^{an})(K^{sh})\rightarrow (X\times_L
K)^{an}(K^{sh})$$ is a bijection. By \cite[5.2.1]{conrad}, the
source of this map is canonically isomorphic to $(Sm(X\times_L
K))^{an}$, so it is enough to show that $$Sm(X\times_L
K)(K^{sh})\rightarrow (X\times_L K)(K^{sh})$$ is a bijection. This
follows from \cite[3.6.10]{neron} by the same arguments as above.
\end{proof}
\begin{lemma}\label{propbounded}
If $f:Y\rightarrow X$ is a proper morphism of separated rigid
$K$-varieties, then $Y$ is bounded if $X$ is bounded. The same is
true if $f:Y\rightarrow X$ is a proper morphism of $K$-varieties.
\end{lemma}
\begin{proof}
Since $K$ is a discretely valued field, the analytification of a
proper morphism of $K$-varieties is a proper map of separated
rigid $K$-varieties by the concluding remarks in
\cite[\S\,5.2]{conrad}, so we only have to prove the result in the
rigid analytic setting. There it follows from the fact that the
inverse image of a quasi-compact open subvariety under a proper
morphism is again quasi-compact.
\end{proof}

 Let $X$ be a variety over an arbitrary field $F$. A compactification of $X$ is a
dense open immersion $X\rightarrow \overline{X}$ of $X$ into a
proper $F$-variety $\overline{X}$. Such a compactification always
exists by Nagata's embedding theorem. We denote by $\partial
\overline{X}$ the complement of $X$ in $\overline{X}$ (with its
reduced closed subscheme structure). If $X$ is smooth and $F$ has
characteristic zero, then $X$ admits a smooth compactification by
Hironaka's resolution of singularities.

\begin{prop}\label{nopoints}
Let $L$ be a discretely valued field, and let $X$ be a smooth
$L$-variety. We assume that $X$ admits a smooth compactification.
The following properties are equivalent:
\begin{enumerate}
\item $X$ is bounded \item there exists a compactification
$\overline{X}$ of $X$ such that $\partial
\overline{X}(L^{sh})=\emptyset$ \item for every smooth
compactification $\overline{X}$ of $X$, $\partial
\overline{X}(L^{sh})=\emptyset$
\end{enumerate}
\end{prop}
So, in characteristic zero, boundedness means that there are ``no
unramified points at infinity''.
\begin{proof}
The implication $(2)\Rightarrow (1)$ was shown in
\cite[1.1.10]{neron} (only assuming that the ring of integers of
$L$ is excellent) and $(3)\Rightarrow (2)$ follows from our
assumption. So let us prove $(1)\Rightarrow (3)$. By Proposition
\ref{corcompar} we may assume that $L$ is complete. We denote its
ring of integers by $R$ and its residue field by $k$. Let
$\overline{X}$ be any smooth compactification of $X$.
Let $\X$ be a weak N\'eron model for $\overline{X}^{an}$. By
boundedness of $X$ and \cite[4.4]{formrigI} we may assume that
there exists an open formal subscheme $\mV$ of $\X$ such that
$\mV_\eta$ is contained in $X^{an}$, and such that
$\mV_\eta(L')=X^{an}(L')$ for each finite unramified extension
$L'/L$.

Suppose that any closed point $x$ on $\X_s$ whose residue field is
separable over $k$ is contained in $\mV_s$. Then
$$\overline{X}(L^{sh})=\X_\eta(L^{sh})=\mV_\eta(L^{sh})=X(L^{sh})$$ and the lemma is proven. Hence,
we may assume that there exists a closed point $x$ in the
complement of $\mV_s$ in $\X_s$ whose residue field is separable
over $k$. Passing to a finite unramified extension of $R$, we may
suppose that $x\in \X_s(k)$.

The tube $]x[$ of $x$ in $\X$ is an open rigid subvariety of
$\X_\eta$ (see \cite[1.1.2]{bert}), and hence of $X^{an}$. Since
$\X$ is formally smooth over $R$, the map
$$\X(R/\mathfrak{M}^{n+1})\rightarrow \X(R/\mathfrak{M}^{n})$$ is
surjective for each $n\geq 0$, and by completeness of $R$, $x$
lifts to a section in $\X(R)$. Hence, there is an isomorphism of
$R$-algebras
$$\widehat{\mathcal{O}}_{\X,x}\cong R[[y_1,\ldots,y_d]]$$ with
$d=dim(X)$, by \cite[3.1.2]{neron}.
 Moreover, by \cite[0.2.7]{bert}, $]x[$ is
canonically isomorphic to the generic fiber of the special formal
$R$-scheme $\Spf \widehat{\mathcal{O}}_{\X,x}$, which is the open
unit polydisc $\mathbb{B}^d_L$ of dimension $d$ over $L$.
Moreover, by our assumptions, all $L$-valued points of $]x[$ are
contained in $(]x[\cap (\partial \overline{X})^{an})(L)$. However,
since this
 implies that $\mathbb{B}^d_{L}\subset (\partial
\overline{X})^{an}$,
while on the other hand $dim(\partial \overline{X})<d$, we arrive
at a contradiction.
\end{proof}

\begin{remark}
If $L$ is a henselian discretely valued field and $X$ an
irreducible $L$-variety with a smooth $L$-rational point, then
$X(L)$ is dense in $X$. This is well-known and can be proved in an
elementary way; it can also be deduced from the existence of weak
N\'eron models using an argument similar to the one in the proof
of $(1)\Rightarrow (3)$.
\end{remark}


\begin{definition}[Weak N\'eron model of a rigid variety \cite{formner}, Def.\,1.3]
Let $X$ be a separated rigid $K$-variety. A weak N\'eron model for
$X$ is a smooth separated formal $R$-scheme $\X$, topologically of
finite type, endowed with an open immersion $h:\X_\eta\rightarrow
X$, such that $h$ induces a bijection $\X_\eta(K')\rightarrow
X(K')$ for each finite unramified extension $K'/K$.
\end{definition}
\begin{prop}
A separated rigid $K$-variety $X$ admits a weak N\'eron model iff
$X$ is smoothly bounded.
\end{prop}
\begin{proof}
This condition is obviously necessary. It is also sufficient:
observe that, if $V$ is a smooth quasi-compact open subvariety of
$X$ with $V(K^{sh})=X(K^{sh})$, a weak N\'eron model for $V$ is
also a weak N\'eron model for $X$, and apply \cite[3.3]{formner}.
\end{proof}

We establish some elementary properties of weak N\'eron models
which we'll need in the following section.
\begin{lemma}\label{compl}
Let $S$ be a discrete valuation ring, with maximal ideal
$\mathfrak{N}$ and quotient field $L$, and let $R$ be its
completion. Let $X$ be a smooth and bounded $L$-variety, and let
$Y$ be a smooth $S$-variety endowed with an isomorphism
$f:Y_L\rightarrow X$ such that $(Y,f)$ is a weak N\'eron model for
$X$. We put $Y_K=Y_L\times_L K$. Denote by $\mY\rightarrow \Spf R$
the formal $\mathfrak{N}$-adic completion of $Y\rightarrow \Spec
S$ and by $h$ the composition
$$\begin{CD}
\mY_\eta@>>> (Y_K)^{an}@>f^{an}>> (X\times_L K)^{an}
\end{CD}$$
where the first arrow is the canonical open immersion
\cite[0.3.5]{bert}. Then $(\mY,h)$ is a weak N\'eron model for
$(X\times_L K)^{an}$.
\end{lemma}
\begin{proof}
We only have to show that the canonical open immersion
$\mY_\eta\rightarrow (Y_K)^{an}$
 induces a
bijection $\mY_\eta(K')\rightarrow (Y_K)^{an}(K')$, for any finite
unramified extension $K'/K$. By definition of the analytification
functor $(.)^{an}$ (see e.g. \cite[0.3.3]{bert}), there is a
natural map of locally ringed sites $(Y_K)^{an}\rightarrow Y_K$
which induces a canonical bijection $Y_K(K')=(Y_K)^{an}(K')$.
Moreover, since $Y$ is a weak N\'eron model for $Y_L$, it follows
from \cite[3.6.7]{neron} that $Y\times_S R$ is a weak N\'eron
model for $Y_K$, so the natural map $Y(R')\rightarrow Y_L(K')$ is
a bijection, with $R'$ the normalization of $R$ in $K'$. Hence,
the result follows from the canonical bijections
$\mY_\eta(K')=\mY(R')=Y(R')$.
\end{proof}

\begin{lemma}\label{prod}
If $X,\,Y$ are smoothly bounded rigid varieties over $K$, and if
$(\X,f)$ and $(\mY,g)$ are weak N\'eron models of $X$, resp. $Y$,
then
$$(\X\times_R \mY,\,f\times_K g:\X_\eta\times_K
\mY_\eta\rightarrow X\times_K Y)$$ is a weak N\'eron model for
$X\times_K Y$.
\end{lemma}
\begin{proof}
Since smoothness is preserved under base-change, and the
composition of two smooth morphisms is again smooth, we see that
$\X\times_R \mY$ is a smooth $stft$ formal $R$-scheme. Note also
that the fibered product commutes with taking generic fibers
\cite[4.6]{formrigI}, so that the generic fiber of $\X\times_R
\mY$ is canonically isomorphic to $\X_\eta\times_K \mY_\eta$. As a
fiber product of two open immersions, the morphism $f\times_K g$
is again an open immersion. It follows immediately from the
universal property of the fiber product that $(\X\times_R
\mY,f\times_K g)$ is a weak N\'eron model for $X\times_K Y$.
\end{proof}

\section{Motivic Serre invariants for algebraic varieties}
In this section, we assume that $R$ is complete, and that the
residue field $k$ of $R$ is perfect.

\begin{definition}[Motivic Serre invariant]
Let $X$ be a smoothly bounded rigid $K$-variety, and let $(\X,h)$
be a weak N\'eron model for $X$. We define the motivic Serre
invariant $S(X)$ of $X$ by
$$S(X)=[\X_s]\ \in K_0(Var_k)/(\LL-1)$$
This invariant only depends on $X$, and not on the choice of a
weak N\'eron model.

If $Y$ is a smoothly bounded $K$-variety, then the associated
rigid $K$-variety $Y^{an}$ is smoothly bounded by Proposition
\ref{comparbounded}, so $S(Y^{an})$ is well-defined, and we put
$$S(Y)=S(Y^{an})\ \in K_0(Var_k)/(\LL-1)$$
\end{definition}

The fact that $S(X)$ only depends on $X$, and not on the choice of
a weak N\'eron model was proven in \cite[4.5.3]{motrigid} for $X$
smooth and quasi-compact, using the theory of motivic integration
on formal schemes, and in \cite[5.11]{NiSe-weilres} for $X$ smooth
and bounded. The proof of \cite[5.11]{NiSe-weilres} also applies
to the case where $X$ is smoothly bounded. Note that $S(X)=0$ if
$X(K^{sh})=\emptyset$, and more generally, $S(X)=S(X')$ if $X$ is
a bounded open rigid subvariety of $X$ such that
$X(K^{sh})=X'(K^{sh})$.

\begin{lemma}\label{weaknerser}
Let $L$ be a discretely valued field, with perfect residue field
$k$, and let $K$ be its completion. If $X$ is a smooth and bounded
$L$-variety, and $(Y,f)$ is a weak N\'eron model for $X$, then
$$S(X\times_L K)=[Y_s]\ \in K_0(Var_k)/(\LL-1)$$
In particular, this value only depends on $X\times_L K$ and not on
the chosen weak N\'eron model.
\end{lemma}
\begin{proof}
This follows immediately from Lemma \ref{compl}.
\end{proof}
\begin{lemma}\label{blup}
Let $X$ be a smooth and bounded $K$-variety, and $A$ a closed
subvariety of $X$, smooth over $K$. Denote by $h:X'\rightarrow X$
the blow-up of $X$ at $A$, and by $E$ the exceptional divisor
$h^{-1}(A)$. Then
$$S(X')-S(E)=S(X)-S(A)$$ in $K_0(Var_k)/(\LL-1)$.
\end{lemma}
\begin{proof}
Let $((Y,B),f)$ be a weak N\'eron model for $(X,A)$. By Lemma
\ref{weaknerser}, $S(X)=[Y_s]$ and $S(A)=[B_s]$ in
$K_0(Var_k)/(\LL-1)$. Denote by $h:Y'\rightarrow Y$ the blow-up of
$Y$ at $B$, and by $F=h^{-1}(B)$ the exceptional divisor. Since
$B$ is smooth over $R$, $Y'$ and $F$ are also smooth over $R$.
Moreover, since blowing up commutes with flat base change, the
isomorphism
$$f:(Y_K,B_K)\rightarrow (X,A)$$ induces an isomorphism
$$f':(Y'_K,F_K)\rightarrow (X',E)$$

We'll show that $((Y',F),f')$ is a weak N\'eron model for
$(X',E)$. We only have to prove that any $K^{sh}$-valued point $x$
on $Y'_K$ extends to a section in $Y'(R^{sh})$. Since $(Y,B)$ is a
weak N\'eron model for $(X,A)$, the point $h(x)\in Y_K(K^{sh})$
extends to a section $a$ in $Y(R^{sh})$. But $h$ is proper, so $x$
itself extends to a section in $Y'(R^{sh})$.

This implies that $S(X')=[Y'_s]$ and $S(E)=[F_s]$, and since $h$
restricts to an isomorphism $Y'-F\rightarrow Y-B$, we have
$[Y'_s]-[F_s]=[Y_s]-[B_s]$ in $K_0(Var_k)$, so the result follows.
\end{proof}

\begin{theorem}\label{serresing}
Assume that $K$ has characteristic zero. There exists a unique
ring morphism
$$S:\mathcal{M}_K \rightarrow K_0(Var_k)/(\LL-1)$$
such that $S([X])=S(X)$ for any smooth and proper $K$-variety $X$.
It satisfies $S([X])=S(X)$ for any smoothly bounded $K$-variety
$X$, and $S(\LL-1)=0$.
\end{theorem}
\begin{proof}
By Theorem \ref{bittner} and Lemma \ref{blup}, there exists a
unique morphism of abelian groups
$$S:K_0(Var_K)\rightarrow K_0(Var_k)/(\LL-1)$$
such that $S([X])=S(X)$ for any smooth and proper $K$-variety $X$.
By Lemma \ref{prod}, and the fact that the analytification functor
$(\cdot)^{an}$ commutes with fiber products, $S$ is a morphism of
rings. We have
$$S(\LL)=S(\mathbb{P}^1_K)-S(\Spec K)=\LL=1$$ in $K_0(Var_k)/(\LL-1)$, so
$S$ localizes to a ring morphism on $\mathcal{M}_K$ and
$S(\LL-1)=0$. It remains to show that $S([X])=S(X)$ if $X$ is
smoothly bounded. We proceed by induction on the dimension of $X$.

 If $X$ has
dimension $0$, then $X$ is proper and smooth over $K$, so
$S([X])=S(X)$ by definition. Suppose that $dim(X)>0$. Since $K$
has characteristic zero and $X$ is reduced, the $K$-smooth locus
$Sm(X)$ of $X$ is open and dense in $X$. But $X$ is smoothly
bounded, so $(X-Sm(X))(K^{sh})$ is empty (and in particular,
$X-Sm(X)$ is smoothly bounded). Since
$$dim(X-Sm(X))<dim(X)$$ we know that $S([X-Sm(X)])=S(X-Sm(X))=0$
by the induction hypothesis.
 By additivity, $S([X])=S([Sm(X)])$, so we may assume that $X$ is smooth over $K$.

We embed $X$ as a dense open subscheme in a smooth proper
$K$-variety $\overline{X}$, and we denote the boundary
$\overline{X}-X$ by $\partial \overline{X}$. Since $X$ is bounded,
we know that $\partial \overline{X}(K^{sh})=\emptyset$, by
Proposition \ref{nopoints}. Again by our induction hypothesis,
this implies that $S([\partial \overline{X}])=0$, so
$$S([X])=S([\overline{X}])=S(\overline{X})=S(X)$$ as required.
\end{proof}
\begin{definition}[Motivic Serre invariant of an algebraic variety]\label{def-motserre}
Assume that $K$ has characteristic zero. For any separated
$K$-scheme of finite type $X$, we define the motivic Serre
invariant $S(X)$ of $X$ as the image of $[X]$ under the morphism
$$S:K_0(Var_K)\rightarrow K_0(Var_k)/(\LL-1)$$
\end{definition}
\begin{example}
If $X$ is the cusp $\Spec K[x,y]/(x^2-y^3)$, then $X^{an}$ is not
bounded, so $X^{an}$ does not admit a weak N\'eron model (and
neither does $X$). However, we can break up $X$ into the disjoint
union of the origin $O$ and its complement
$$Y=\Spec K[x,x^{-1},y,y^{-1}]/(x^2-y^3)$$
Since $Y$ is isomorphic to the torus $\mathbb{G}_{m,K}$, we get
$S(X)=1$ in $K_0(Var_k)/(\LL-1)$. Alternatively, we can use the
fact that the normalization of $Y$ is isomorphic to $\A^1_K$ and
that the inverse image of the singular point in this normalization
consists of a unique $K$-point.
\end{example}
\begin{example}
Let $X$ be a rational projective curve with $\delta$ nodes, and no
other singularities, and suppose that all the nodes and their
tangent directions are rational over $K$. Then the normalization
$\widetilde{X}$ is isomorphic to $\mathbb{P}^1_K$, and over each
node of $X$ lie exactly $2$ points of $\widetilde{X}$, which are
$K$-rational. Hence,
$$S(X)-\delta=S(\widetilde{X})-2\delta$$ whence $S(X)=2-\delta$ in
$K_0(Var_k)/(\LL-1)$.
\end{example}

\begin{lemma}\label{empty}
Assume that $K$ has characteristic zero, and let $X$ be a variety
over $K$. If $X(K^{sh})=\emptyset$, then $S(X)=0$, and if $k=k^s$
and $X(K)$ is finite, then $S(X)=\sharp X(K)$.
\end{lemma}
\begin{proof}
By additivity, it suffices to prove the result when
$X(K^{sh})=\emptyset$. Then $X$ is smoothly bounded, so
$S(X)=S(X^{an})$. But $S(X^{an})=0$ since the empty formal scheme
is a weak N\'eron model for $X^{an}$.
\end{proof}


\section{The trace formula}\label{sec-trace}
In this section, we assume that $R$ is complete and $k$
algebraically closed, and we fix a prime number $\ell$ invertible
in $k$. For each integer $d>0$ prime to the characteristic
exponent $p$ of $k$, we denote by $K(d)$ the unique extension of
degree $d$ of $K$ in a fixed separable closure $K^s$. We denote by
$K^t$ the tame closure of $K$ in $K^s$.

 For any pro-finite
group $H$, we denote by $Rep_H(\Q_\ell)$ the abelian tensor
category of $\ell$-adic representations of $H$ (i.e. finite
dimensional $\Q_\ell$-vector spaces endowed with a continuous left
action of $H$) and by $K_0(Rep_H(\Q_\ell))$ its Grothendieck ring.
For each element $h$ of $H$, we consider the unique ring morphism
$$Tr_{h}:K_0(Rep_{H}(\Q_\ell))\rightarrow \Q_\ell$$
mapping $[M]$ to $Trace(h\,|\,M)$ for any $\ell$-adic
representation $M$ of $H$.


Denote by $G_K$ the monodromy group $G(K^s/K)$. Consider the
\'etale realization morphism
$$\acute{e}t:K_0(Var_K)\rightarrow K_0(Rep_{G_K}(\Q_\ell))$$ from
Section \ref{sec-real}.
%
%

If we denote by $G^t_K$ the tame monodromy group $G(K^t/K)$, then
there is a natural surjective morphism $G_K\rightarrow G^t_K$
whose kernel is the wild inertia group $P$. This morphism induces
a canonical morphism of rings
$$K_0(Rep_{G^t_K}(\Q_\ell))\rightarrow K_0(Rep_{G_K}(\Q_\ell))$$
Since $P$ is a pro-$p$-group and $\ell$ is prime to $p$, the
functor $$(\cdot)^P:Rep_{G_K}(\Q_\ell)\rightarrow
Rep_{G_K^t}(\Q_\ell)$$ is exact, so it defines a morphism of
abelian groups
$$(\cdot)^P:K_0(Rep_{G_K}(\Q_\ell))\rightarrow
K_0(Rep_{G^t_K}(\Q_\ell))$$ which is left inverse to
$$K_0(Rep_{G^t_K}(\Q_\ell))\rightarrow K_0(Rep_{G_K}(\Q_\ell))$$
Hence, the latter morphism is injective, and we may identify
$K_0(Rep_{G^t_K}(\Q_\ell))$ with its image in
$K_0(Rep_{G_K}(\Q_\ell))$. Then an element $\alpha$ of
$K_0(Rep_{G_K}(\Q_\ell))$ belongs to $K_0(Rep_{G^t_K}(\Q_\ell))$
iff $(\alpha)^P=\alpha$.
%
\begin{definition}[Tame varieties]\label{def-tame}
If $X$ is a smooth and proper $K$-variety, then we say that $X$ is
tame if there exists a regular proper $R$-variety $Y$ such that
$Y_s$ is a tame strict normal crossings divisor (i.e. the
multiplicity of each component is prime to $p$) and such that
$Y_K$ is isomorphic to $X$. Such a model $Y$ will be called a tame
$R$-model for $X$.

The tame Grothendieck ring of varieties over $K$ is the subring
$K_0^{t}(Var_K)$ of $K_0(Var_K)$ generated by the isomorphism
classes $[X]$ of tame smooth proper $K$-varieties $X$.
\end{definition}
Of course, if $k$ has characteristic zero, then any smooth and
proper $K$-variety is tame, and $K_0^{t}(Var_K)=K_0(Var_K)$.
\begin{lemma}\label{tame}
The image of the \'etale realization morphism
$$\acute{e}t:K_0^{t}(Var_K)\rightarrow K_0(Rep_{G_K}(\Q_\ell))$$ is contained in
$K_0(Rep_{G^t_K}(\Q_\ell))$.  If $X$ is a $K$-variety such that
$[X]$ belongs to $K_0^t(Var_K)$, then
$$\acute{e}t(X)=\sum_{i\geq 0}(-1)^i[H^i_c(X\times_K K^t,\Q_\ell)]$$
in $K_0(Rep_{G^t_K}(\Q_\ell))$.
\end{lemma}
\begin{proof}
 If $X$ is a tame, smooth
and proper $K$-variety, and $Y$ is a tame $R$-model for $X$, then
by \cite[2.23]{Rapo-Zink}, the $\ell$-adic nearby cycles complex
$R\psi_\eta(\Q_\ell)$ of $Y$ is tame, i.e. $P$ acts trivially on
$R^i\psi_\eta(\Q_\ell)$ for each $i\geq 0$. By the spectral
sequence \cite[I.2.2.3]{sga7a} this implies that $P$ acts
trivially on $H^i(X\times_K K^s,\Q_\ell)$, for each $i\geq 0$.
Since the isomorphism classes of tame smooth and proper
$K$-varieties $X$ generate the subring $K_0^t(Var_K)$ of
$K_0(Var_K)$, we see that the image of \'etale realization
morphism
$$\acute{e}t:K_0^{t}(Var_K)\rightarrow K_0(Rep_{G_K}(\Q_\ell))$$ is contained in
$K_0(Rep_{G^t_K}(\Q_\ell))$.

 Since $\ell$ is invertible in $k$, and
$P$ is a pro-$p$-group, there is a canonical isomorphism
$$H^i_c(X\times_K K^t,\Q_\ell)\cong H^i_c(X\times_K K^s,\Q_\ell)^P$$
for any $K$-variety $X$ and each $i\geq 0$. Hence,
$$\acute{e}t(X)^P=\sum_{i\geq 0}(-1)^i[H^i_c(X\times_K
K^t,\Q_\ell)]$$ in $K_0(Rep_{G_K^t}(\Q_\ell))$. If $[X]$ belongs
to $K_0^t(Var_K)$ then $\acute{e}t(X)^P=\acute{e}t(X)$ by the
first part of the proof, and we are done.
\end{proof}
\begin{prop}[Trace formula for tame varieties]\label{tame-3}
Let $\varphi$ be a topological generator of the tame monodromy
group $G^t_K$. If $X$ is a tame smooth and proper $K$-variety,
then
$$\chi_{top}(S(X\times_K K(d)))=Trace(\varphi^d\,|\,H(X\times_K K^t,\Q_\ell))$$
for each integer $d>0$ prime to $p$.
\end{prop}
\begin{proof}
This follows immediately from the trace formula in
\cite[5.4]{NiSe} and the comparison theorem for \'etale cohomology
\cite[7.5.4]{Berk-etale}. See also \cite[5.4]{NiSe} for an
explicit expression in terms of a tame $R$-model of $X$.
%
%
\end{proof}

%


\begin{theorem}[Trace formula]\label{trace}
Assume that $K$ has characteristic zero. If $d>0$ is an integer
prime to $p$ and $\varphi$ is a topological generator of the tame
Galois group $G^t_{K(d)}=G(K^t/K(d))$, then the following diagram
of ring morphisms commutes:
$$\xymatrix{
K_0^{t}(Var_{K}) \ar[rd]_{\acute{e}t} \ar[r] &K_0(Var_{K(d)})\ar[r]^-{S} & \ar[d]^{\chi_{top}} K_0(Var_k)/(\LL-1)\\
& K_0(Rep_{G^t_K}(\Q_\ell)) \ar[r]_-{Tr_{\varphi^d}} &\Q_\ell}$$
%
(the upper left horizontal morphism is the natural base change
morphism). In particular, for any $K$-variety $X$ such that $[X]$
belongs to $K_0^{t}(Var_K)$, we have
$$\chi_{top}(S(X\times_K K(d)))=Trace(\varphi^d\,|\,H_c(X\times_K K^t,\Q_\ell))$$
\end{theorem}
\begin{proof}
 Since the classes $[X]$
of tame smooth proper $K$-varieties generate $K_0^{t}(Var_K)$,
this follows from Proposition \ref{tame-3} and Lemma \ref{tame}.
\end{proof}
\begin{cor}\label{char0}
If $k$ has characteristic zero, then for any $K$-variety $X$,
$$\chi_{top}(S(X))=Trace(\varphi\,|\,H_c(X\times_K K^s,\Q_\ell))$$
\end{cor}
\begin{cor}\label{existpoint}
If $k$ has characteristic zero, and if $X$ is a $K$-variety, then
$X$ has a rational point iff there exists a subvariety $U$ of $X$
such that
$$Trace(\varphi\,|\,H(U\times_K K^s,\Q_\ell))\neq 0$$
\end{cor}
\begin{proof}
The ``if'' part follows from Lemma \ref{empty} and Corollary
\ref{char0}. For the converse implication we can take for $U$ a
rational point on $X$.
\end{proof}
There are examples of (non-tame) smooth and proper $K$-varieties
$X$ such that
$$\chi_{top}(S(X))\neq Trace(\varphi\,|\,H_c(X\times_K
K^t,\Q_\ell))$$ The following elementary example was given in
\cite[\S\,5]{NiSe}: let $R$ be the ring of Witt vectors
$W(\mathbb{F}_p^s)$ over the algebraic closure of a finite field
$\mathbb{F}_p$ of characteristic $p$, and put $X=\Spec
K[T]/(T^p-p)$. Then $X$ is smooth and proper over $K$, and since
$X(K)=\emptyset$, we have $S(X)=0$. On the other hand,
$H^i(X\times_K K^t,\Q_\ell)=0$ for $i>0$, and $H^0(X\times_K
K^t,\Q_\ell)$ is isomorphic to $\Q_\ell$ with the trivial
$G^t_K$-action, so that
$$Trace(\varphi\,|\,H(X\times_K K^t,\Q_\ell))=1$$
Of course, it would be very interesting to obtain a cohomological
interpretation of $\chi_{top}(S(X))$ in terms of $\acute{e}t(X)$
if $X$ is not tame,
already in the case where $X$ is smooth and proper over $K$. We
will see below that this is not always possible (Proposition
\ref{counterex}).

\begin{definition}[Error term]
Let $\varphi$ be a topological generator of the tame Galois group
$G^t_K$. If $X$ is any smooth and proper $K$-variety, we put
$$e(X)=Trace(\varphi\,|\,H(X\times_K
K^t,\Q_\ell))-\chi_{top}(S(X))$$
We say that the trace formula holds for $X$ iff $e(X)=0$.
\end{definition}
In particular, by Corollary \ref{char0}, the trace formula holds
for any $K$-variety $X$ if $k$ has characteristic zero.
\section{Trace formula for curves}\label{sec-curves}
In this section, we assume that $R$ is complete and $k$ is
algebraically closed. We fix a prime $\ell$ invertible in $k$. We
denote by $\varphi$ a topological generator of the tame Galois
group $G(K^t/K)$, and by $P\subset G(K^s/K)$ the wild inertia
group.

\begin{definition}[Cohomological tameness]
If $X$ is a $K$-variety, we say that $X$ is cohomologically tame
if $P$ acts trivially on $H^i_c(X\times_K K^s,\Q_\ell)$ for each
$i\geq 0$.
\end{definition}

If $X$ is a tame smooth proper $K$-variety, then $X$ is
cohomologically tame (cf. proof of Lemma \ref{tame}).
 We will study the validity of the
trace formula for smooth, proper, geometrically connected curves
over $K$, and we will see that there are remarkable connections
with T. Saito's criterion for cohomological tameness \cite{saito}.

\subsection{A general result for curves}
If $Y$ is a regular $R$-variety and $Y_s$ is a normal crossings
divisor, we denote the irreducible components of $(Y_s)_{red}$ by
$E_i,\,i\in I$, and we denote by $N_i$ the multiplicity of $E_i$
in $Y_s$. We write $Y_s=\sum_{i\in I}N_i E_i$ as usual. For each
$i\in I$, we put
$$E_i^o=E_i\setminus (\cup_{j\neq i}E_j)$$
and we denote by $Sm(E_i^o)$ its $k$-smooth locus. If $Y_s$ has
strict normal crossings, then $Sm(E_i^o)=E_i^o$.
\begin{definition}[Wild locus]
Let $Y$ be a regular $R$-variety such that $Y_s$ is a normal
crossings divisor. If $k$ has characteristic $p>0$, then we define
the wild locus $W_Y$ of $Y$ as the disjoint union of the
subvarieties $Sm(E_i^o)$ of $Y$ with $N_i=p^{e_i}$ for some
$e_i>0$. If $k$ has characteristic zero, we put $W_Y=\emptyset$.
\end{definition}
\begin{theorem}\label{general}
Let $X$ be a smooth and proper curve over $K$, and let $Y$ be a
regular $R$-model  for $X$ such that $Y_s$ has strict normal crossings. 
Then
$$e(X)=\chi_{top}(W_Y)$$
 so the trace formula holds for $X$ iff
$\chi_{top}(W_Y)=0$.
\end{theorem}
\begin{proof}
If $y$ is a closed point of $Y_s$, then the computation of the
tame nearby cycles in \cite[I.3.3]{sga7a} shows that
$$
Trace(\varphi\,|\,R\psi_\eta^t(\Q_\ell)_y)=\left\{\begin{array}{l}
                                                0~\textrm{if $y\in
Sm(Y_s)\cup W_Y$}\\
                                                    {}\\
                                                1~\textrm{else.}\\
                                                        \end{array}\right.
$$
Moreover, by \cite[3.3]{abbes}, the complex
$R\psi_\eta^t(\Q_\ell)$ is tamely constructible (in the sense of
\cite{Ni-trace}), so \cite[6.3]{Ni-trace} applies and
\begin{eqnarray*}
Trace(\varphi\,|\,H(X\times_K
K^t,\Q_\ell))&=&Trace(\varphi\,|\,\mathbb{H}(Y_s,R\psi_\eta^t(\Q_\ell)))
\\&=&\chi_{top}(Sm(Y_s))+\chi_{top}(W_Y)
\end{eqnarray*}
Since $Y$ is regular, $Sm(Y)$ is a weak N\'eron model for $X$ (cf.
remark following \cite[3.1.2]{neron}) so
$\chi_{top}(S(X))=\chi_{top}(Sm(Y_s))$.
\end{proof}
\subsection{Curves of genus $\neq 1$}
\begin{theorem}\label{global}
Let $X$ be a proper smooth geometrically connected curve over $K$
of genus $g\neq 1$, and assume that $X$ is cohomologically tame.
Then the trace formula holds for $X$.
\end{theorem}
\begin{proof}
In view of Corollary \ref{char0}, we may suppose that $k$ has
characteristic $p>0$. Let $Y$ be a relatively minimal regular
$R$-model with normal crossings of $X$ ($RMN$-model in the
terminology of \cite[3.1.1]{saito}), with $Y_s=\sum_{i\in I}E_i$.
Then by Saito's criterion \cite[3.11]{saito}, the fact that $X$ is
cohomologically tame implies that $E_i^o$ is smooth and
$\chi_{top}(E_i^o)=0$ if $p$ divides $N_i$, so $\chi_{top}(W_Y)=0$
and we may conclude by Theorem \ref{general}.
\end{proof}

\subsection{Elliptic curves}\label{sec-elliptic}
\begin{theorem}\label{elliptic}
Let $X$ be an elliptic curve over $K$.
\begin{itemize}
\item $X$ has multiplicative reduction iff $S(X)=0$ \item $X$ has
additive reduction iff $S(X)\in \{1,2,3,4\}$. In this case,
$S(X)=n$, with $n$ the number of connected components of the
special fiber of the N\'eron minimal model of $X$. More precisely:

$S(X)=1$ iff $X$ is of type $II$ or $II^*$;

 $S(X)=2$ iff $S(X)$ is
of type $III$ or $III^*$;

 $S(X)=3$ iff $X$ is of type $IV$ or
$IV^*$;

  $S(X)=4$ iff $X$ is of type $I^*_{\nu}$, $\nu\geq 0$.
\item $X$ has good reduction $\overline{X}$ iff
$S(X)\notin\{0,1,2,3,4\}$, and in this case, $S(X)=[\overline{X}]$
\end{itemize}
In particular, $\chi_{top}(S(X))=0$ iff $X$ has semi-stable
reduction. Moreover, the trace formula holds for $X$ iff we're in
one of the following situations:
\begin{itemize} \item $X$ is cohomologically tame, \item $p=2$ and
$X$ is of type $III$ or $III^*$.
\end{itemize}
\end{theorem}
\begin{proof}
By definition, $S(X)=[\mathcal{A}_s]$ in $K_0(Var_k)/(\LL-1)$
where $\mathcal{A}$ is the N\'eron model of $X$. It follows
immediately that $S(X)=0$ if $X$ has multiplicative reduction,
$S(X)=[\overline{X}]$ if $X$ has good reduction $\overline{X}$ and
$S(X)=n$ if $X$ has additive reduction, with $n$ the number of
connected components of $\mathcal{A}_s$. The values for $n$ can be
read from the Kodaira-N\'eron reduction table (see e.g.
\cite[IV.9]{silver-adv}). We only have to check that
$[\overline{X}]\notin \{0,1,2,3,4\}\subset K_0(Var_k)/(\LL-1)$ if
$X$ has good reduction $\overline{X}$. However, for any elliptic
curve $E$ over $k$, its Poincar\'e polynomial $P(E;T)$ is equal to
$1+2T+T^2$ which is not congruent to any integer modulo
$P(\LL-1;T)=T^2-1$.

By Theorem \ref{general}, Saito's criterion \cite[3.11]{saito} and
direct computation on the reduction table, we see that the trace
formula holds for $X$ iff we're in one of the two cases described
in the statement (for a more precise analysis, see below).
\end{proof}

Let us investigate the cases where $X$ is not cohomologically
tame. By Saito's criterion \cite[3.11]{saito} this happens exactly
in the following situations:
\begin{enumerate}
\item $k$ has characteristic $2$, and $X$ has type $II$, $II^*$,
$III$, $III^*$, or $I_{\nu}^*$, $\nu\geq 0$. \item $k$ has
characteristic $3$, and $X$ has type $II$, $II^*$, $IV$ or $IV^*$.
\end{enumerate}
Using the expression for $e(X)$ in Theorem \ref{general}, we can
read the following values from the reduction table:
\begin{enumerate}
\item Suppose that $k$ has characteristic $2$. If $X$ has type
$II$ or $II^*$, then $e(X)=1$. If $X$ has type $III$ or $III^*$,
then $e(X)=0$ and the trace formula holds. If $X$ has type
$I_{\nu}^*$, $\nu\geq 0$, then $e(X)=-2$. \item Suppose that $k$
has characteristic $3$. If $X$ has type $II$ or $II^*$ then
$e(X)=1$. If $X$ has type $IV$ or $IV^*$ then $e(X)=-1$.
\end{enumerate}

\begin{remark}
It seems reasonable to expect that the trace formula holds for all
cohomologically tame abelian varieties $A$. If $k$ has
characteristic zero, this follows from Corollary \ref{char0}. If
$k$ has positive characterstic, the trace formula holds if $A$
does not have purely additive reduction, and also if $A$ is the
Jacobian of a curve. Details and further results will appear in a
forthcoming paper.
\end{remark}

\subsection{Curves of genus $1$ without rational point}
Finally, we discuss the case of curves of genus $1$ without
rational point. Let $X$ be a smooth, proper, geometrically
connected $K$-curve of genus $1$. Then its Jacobian $Jac(X)$ is an
elliptic curve. If we denote by $m(X)$ the order of the torsor $X$
in the group $H^1(K,Jac(X))$, then the reduction type of $X$ is
equal to $m(X)$ times the reduction type of $Jac(X)$, by
\cite[6.6]{liu-lorenzini-raynaud} (i.e. the multiplicities of the
components of the reduction are multiplied by $m(X)$).
\begin{theorem}\label{noratpoint}
Let $X$ be a smooth, proper, geometrically connected $K$-curve of
genus $1$, and assume that $X(K)$ is empty. Then $S(X)=0$, and
$$e(X)=\chi_{top}(S(Jac(X)))+e(Jac(X))$$

The trace formula holds for $X$ iff

1. $k$ has characteristic $0$, or

2. $k$ has characteristic $p>0$ and $Jac(X)$ has semi-stable
reduction.
%
%
\end{theorem}
\begin{proof}
The fact that $X(K)$ is empty implies that $S(X)=0$, since the
empty scheme is a weak N\'eron model for $X$. Moreover, there
exists a canonical $G^t_K$-equivariant isomorphism
$$H(X\times_K K^t,\Q_\ell)\cong H(Jac(X)\times_K K^t,\Q_\ell)$$
so that
$$Trace(\varphi\,|\,H(X\times_K K^t,\Q_\ell))=Trace(\varphi\,|\,H(Jac(X)\times_K
K^t,\Q_\ell))$$ and
$$e(X)=\chi_{top}(S(Jac(X)))+e(Jac(X))$$
Hence, the trace formula holds for $X$ iff
$$\chi_{top}(S(Jac(X)))+e(Jac(X))=0$$
We know from Corollary \ref{char0} that the trace formula holds
for $X$ if $k$ has characteristic zero, so assume that $k$ has
characteristic $p>0$. The computations in Section \ref{elliptic}
show that $$\chi_{top}(S(Jac(X)))+e(Jac(X))=0$$ iff $Jac(X)$ has
semi-stable reduction.
\end{proof}

\begin{prop}\label{counterex}
If $k$ has characteristic $p>0$, then there exists a smooth,
proper, geometrically connected curve $X$ over $K$ of genus $1$
such that $X$ is cohomologically tame and such that the trace
formula does not hold for $X$.
\end{prop}
\begin{proof}
Choose a cohomologically tame elliptic curve $E$ over $K$ such
that $E$ has additive reduction. This is possible for any value of
$p$, by Saito's citerion \cite[3.11]{saito}.
Since $k$ is algebraically closed and $K$ is complete, we have
$H^1(K,E)\neq 0$ (as noted in \cite[6.7]{liu-lorenzini-raynaud}
this follows from the results in \cite{begueri} in the mixed
characteristic case, and from those in \cite{bertapelle} in the
equicharacteristic case). Any non-zero element in $H^1(K,E)$
corresponds to a smooth, proper, geometrically connected curve $X$
over $K$ of genus $1$ without rational point, whose Jacobian is
isomorphic to $E$. By the existence of a $G_K$-equivariant
isomorphism
$$H(X\times_K K^s,\Q_\ell)\cong H(E\times_K K^s,\Q_\ell)$$ we
know that $X$ is cohomologically tame. Since the trace formula
holds for $E$, by Theorem \ref{elliptic}, we see that the trace
formula holds for $X$ iff
$$\chi_{top}(S(X))=\chi_{top}(S(E))$$
However, the left hand side vanishes, while the right hand side is
non-zero by Theorem \ref{elliptic}.
%
\end{proof}

The example shows that $\chi_{top}(S(X))$ can, in general, not be
computed from the \'etale realization $\acute{e}t(X)$ (nor even
from the Chow motive with rational coefficients $M(X)$ of $X$)
since $X$ and $Jac(X)$ have the same \'etale realization (and
isomorphic Chow motives \cite[3.3]{Scholl}). We will see below
(proof of Proposition \ref{not-injective}) that, even if $k$ has
characteristic zero, $S(X)$ can in general not be computed from
$M(X)$ (even though $\chi_{top}(S(X))$ can be computed from
$\acute{e}t(X)$ by the trace formula in Corollary \ref{char0}).

Over a finite field $\mathbb{F}_q$, the situation of Proposition
\ref{counterex} does not occur: every smooth, proper,
geometrically connected curve $X$ of genus $1$ over $\mathbb{F}_q$
admits a rational point, since $H^1(\mathbb{F}_q,E)=0$ for every
elliptic curve $E$ over $\mathbb{F}_q$. This result can be
interpreted as a consequence of Grothendieck's trace formula: if
$X$ is a $E$-torsor then $\acute{e}t(X)=\acute{e}t(E)$, so since
$E$ has a rational point the same holds for $X$.

Playing with these ideas, we recover the following classical
result.

\begin{prop}\label{h1}
Let $E$ be an elliptic curve over $K$ with additive reduction.
\begin{enumerate}
\item If $k$ has characteristic zero, then $H^1(K,E)=0$. \item If
$k$ has characteristic $p>0$, then $H^1(K,E)$ is a $p$-group.
\end{enumerate}
\end{prop}
\begin{proof}
$1.$ We know that the trace formula holds if $k$ has
characteristic zero, by Corollary \ref{char0}. Since for any
$E$-torsor $X$, $\acute{e}t(X)=\acute{e}t(E)$, and
$\chi_{top}(S(E))\neq 0$, we conclude that $S(X)\neq 0$, so $X$
has a rational point.

$2.$ Assume that $H^1(K,E)$ contains an element whose order $m$ is
not a power of $p$. It corresponds to a smooth, proper,
geometrically connected curve $X$ of genus $1$, with $Jac(X)\cong
E$. Since the reduction type of $X$ is equal to $m$ times the
reduction type of $E$ \cite[6.6]{liu-lorenzini-raynaud}, we see
that the trace formula holds for $X$, by Theorem \ref{general},
since the wild locus of the minimal regular model with normal
crossings is empty. This contradicts Theorem \ref{noratpoint}.
(For a more direct proof: $Tr(\varphi\,|\,H(X\times_K
K^t,\Q_\ell))=0$ by the computation in the proof of Theorem
\ref{general}; a similar computation shows that
$Tr(\varphi\,|\,H(E\times_K K^t,\Q_\ell))\neq 0$, which is a
contradiction.)
\end{proof}


\begin{prop}\label{not-injective}
Assume that $K$ has characteristic zero.
 The natural ring morphisms
\begin{eqnarray*}
\chi^{eff}&:&K_0(Var_K)\rightarrow K_0(Mot^{eff}_{K})
\\ \chi&:&\mathcal{M}_K\rightarrow K_0(Mot_{K})
\end{eqnarray*}
from Section \ref{sec-chow} are
 both non-injective.
\end{prop}
\begin{proof}
 Let $E$ be an elliptic curve over $K$ with good
reduction, and let $X$ be a non-trivial $E$-torsor. Such a torsor
$X$ exists since $H^1(K,E)\neq 0$ by \cite{shafarevich}. We have
$\chi^{eff}(E)=\chi^{eff}(X)$ by \cite[3.3]{Scholl}, but $S(E)\neq
0$ by Theorem \ref{elliptic} while $S(X)=0$ since $X$ has no
rational point.
 The ring
morphism
$$S:\mathcal{M}_K \rightarrow K_0(Var_k)/(\LL-1)$$ from Theorem \ref{serresing} maps $[X]$
and $[E]$ to $S(X)$, resp. $S(E)$, so $[E]\neq [X]$ in
$\mathcal{M}_K$.
\end{proof}

\begin{prop}\label{not-domain}
Assume that $K$ has characteristic zero. If $A$ is an abelian
variety over $K$ with good reduction, then $[A]$ is a zero divisor
in $K_0(Var_K)$, in $\mathcal{M}_K$, and in $K_0(Var_K)/(\LL-1)$.
\end{prop}
\begin{proof}
let $X$ be a non-trivial $A$-torsor. Such a torsor $X$ exists
since $H^1(K,A)\neq 0$ by \cite{shafarevich}. Then $S(X)= 0$ and
$S(A)=[\overline{A}]\in K_0(Var_k)/(\LL-1)$, with $\overline{A}$
the reduction of $A$. Since the Poincar\'e polynomial
$P(\overline{A};T)$ is not divisible by $P(\LL-1)=T^2-1$ we see
that $S(A)\neq 0$ and therefore $[A]\neq [X]$ in
$K_0(Var_K)/(\LL-1)$.
 However,
$A\times_K A$ and $X\times_K A$ are isomorphic over $K$ and hence
$([A]-[X])\cdot [A]=0$ in $K_0(Var_K)$.
\end{proof}

It is shown in \cite[5.11]{lalu-julienliu} that, more generally,
$[A]$ is a zero-divisor in $K_0(Var_k)$ if $k$ is a field of
characteristic zero and $A$ is an abelian variety over $k$ such
that $H^1(k,A)\neq 0$, but their proof doesn't extend to
$\mathcal{M}_k$.


%

\subsection{The local case}
 Following \cite{saito}, we can also state a local variant of Theorem
 \ref{global}. The category of special formal $R$-schemes is
 defined as in \cite[2.2]{Ni-trace}. The generic fiber $\X_\eta$
 of a special formal $R$-scheme $\X$ is a bounded rigid
 $K$-variety, by \cite[5.8]{NiSe-weilres}. We denote by $\X_0$ the
 reduction of $\X$, i.e. the closed subscheme defined by the
 largest ideal of definition on $\X$.
\begin{theorem}[Local case]\label{local}
Let $X$ be a normal flat $R$-variety of pure relative dimension
$1$, and let $x$ be a closed point of $X_s$ such that $X-\{x\}$ is
smooth over $R$. Denote by $\mathscr{F}_x$ the generic fiber of
the special formal $R$-scheme $\Spf \widehat{\mathcal{O}}_{X,x}$.
If $P$ acts trivially on $H^1(\mathscr{F}_x\widehat{\times}_K
\widehat{K^s},\Q_\ell)$, then
$$\chi_{top}(S(\mathscr{F}_x))=
Trace(\varphi\,|\,H(\mathscr{F}_x\widehat{\times}_K
\widehat{K^t},\Q_\ell))=Trace(\varphi\,|\,R\psi_\eta^t(\Q_\ell)_x)$$
\end{theorem}
Note that $S(\mathscr{F}_x)$ is well-defined since $\mathscr{F}_x$
is a bounded and smooth rigid $K$-variety. Following the
terminology in \cite{NiSe,Ni-trace}, we call $\mathscr{F}_x$ the
analytic Milnor fiber of $X$ at $x$.

\begin{proof}
The second equality follows from \cite[3.5]{berk-vanish2}. By
\cite[4.12]{saito} there exists a proper morphism $h:Y\rightarrow
X$ of $R$-varieties such that $h_K:Y_K\rightarrow X_K$ is an
isomorphism, $Y$ is regular, $Y_s=\sum_{i\in I}N_i E_i$ is a
strict normal crossings divisor, and
$$\chi_{top}(W_Y\cap h^{-1}(x))=0$$
If we denote by $\mZ$ the formal completion of $Y$ along
$h^{-1}(x)$, then $h$ induces an isomorphism $\mZ_\eta\cong
\mathscr{F}_x$ because $h$ is proper.

Since $Y$ is regular, it follows from \cite[3.1.2]{neron} that
$Sm(\mZ)\rightarrow \mZ$ is a special N\'eron smoothening (in the
sense of \cite[4.11]{Ni-trace}) and we see from
\cite[4.14]{Ni-trace} that
$$\chi_{top}(S(\mathscr{F}_x))=\chi_{top}(Sm(\mZ)_0)=\sum_{N_i=1}\chi_{top}(E_i^o\cap
h^{-1}(x))$$

Moreover, there is a canonical $G(K^t/K)$-equivariant isomorphism
$$H(\mZ_\eta\widehat{\times}_K
\widehat{K^t},\Q_\ell)\cong
\mathbb{H}(h^{-1}(x),R\psi^t_Y(\Q_\ell)|_{h^{-1}(x)})$$ by the
comparison results in \cite[3.5]{berk-vanish2}. Now the arguments
in the proof of Theorem \ref{general} show that
$$\chi_{top}(S(\mathscr{F}_x))=
Trace(\varphi\,|\,H(\mathscr{F}_x\widehat{\times}_K
\widehat{K^t},\Q_\ell))$$
\end{proof}

\begin{theorem}\label{localtame}
Let $X$ be a flat, proper, normal $R$-variety of pure relative
dimension $1$ such that $X-Sm(X)$ is a finite set of points, and
such that the $\ell$-adic nearby cycles of $X$ are tame. Then the
trace formula holds for $X_K$.
\end{theorem}
\begin{proof}
By \cite[4.12]{saito} there exists a proper morphism
$h:Y\rightarrow X$ of $R$-varieties such that $h_K:Y_K\rightarrow
X_K$ is an isomorphism, $Y$ is regular, $Y_s=\sum_{i\in I}N_i E_i$
is a strict normal crossings divisor, and $\chi_{top}(W_Y)=0$. Now
the result follows from Theorem \ref{general}.
\end{proof}

\section{Appendix: the Poincar\'e polynomial}\label{Appendix}
Let $k$ be any field. It is, in general, a non-trivial problem to
decide whether the classes of two $k$-varieties $X,\,Y$ in
$K_0(Var_k)$ are distinct. (Larsen and Lunts formulated the
following question in \cite{lars}: does $[X]=[Y]$ in $K_0(Var_k)$
imply that $X$ and $Y$ are piecewise isomorphic? See
\cite{lalu-julienliu} for results in this direction.)

To distinguish elements in $K_0(Var_k)$, it is important to know
some ``computable'' realization morphisms on $K_0(Var_k)$. If $k$
has characteristic zero, we've encountered many of these in the
preceding sections, but in positive characteristic, we're less
equiped. In this section, we'll show how the so-called Poincar\'e
polynomial can be defined over arbitrary base fields by means of a
standard spreading out technique.

We recall the following notation: for any field $k$, any prime
$\ell$ invertible in $k$, and any separated $k$-scheme of finite
type $X$, we denote by $b_i(X)$ the $i$-th $\ell$-adic Betti
number of $X$~:
$$b_i(X)=\mathrm{dim}\,H^i(X\times_k k^s,\Q_\ell)$$
It is known that this value is independent of $\ell$ in the
following cases: \begin{itemize} \item $k$ has characteristic zero
(by comparison with singular cohomology) \item $k$ has
characteristic $p>0$ and $X$ is smooth and proper over $k$ (if $k$
is finite this follows from the cohomological expression for the
zeta function and purity of weight \cite[p.\,27]{katz-intro}; the
general case follows by spreading out to reduce to a finite base
field).
\end{itemize}
To be precise, $b_i(X)$ not only depends on the scheme $X$ but
also on the base field $k$; if we want to make this explicit, we
write $b_i(f)$ instead of $b_i(X)$, with $f:X\rightarrow \Spec k$
the structural morphism.
\subsection{Characteristic zero.}\label{sec-poin0} If $k$ is a field of characteristic zero, there
exists a unique ring morphism $P:K_0(Var_k)\rightarrow \Z[T]$
mapping the class $[X]$ of a smooth and proper $k$-variety to the
polynomial
$$P(X;T)=\sum_{i\geq 0}(-1)^ib_i(X)T^i$$
 Uniqueness and existence follow from Theorem
\ref{bittner}. The morphism $P$ can also be obtained by composing
the Hodge-Deligne realization $HD$ with the ring morphism
$$\Z[u,v]\rightarrow \Z[T]:a(u,v)\mapsto a(T,T)$$

For any element $\alpha$ of $K_0(Var_k)$, we call $P(\alpha)$ the
Poincar\'e polynomial of $\alpha$; for any separated $k$-scheme of
finite type $Y$, we put $P(Y;T)=P([Y])$ and we call this element
of $\Z[T]$ the Poincar\'e polynomial of $Y$. Then
$P(Y;T)=HD(Y;T,T)$, and  $P(Y;1)=HD(Y;1,1)$ is the Euler
characteristic $\chi_{top}(Y)$ of $Y$.

If we write
$$P(Y;T)=\sum_{i\geq 0}(-1)^i\beta_i(Y)T^i$$
then $\beta_i(Y)$ is known as the $i$-th virtual Betti number of
$Y$. If $Y$ is proper and smooth, then $\beta_i(Y)=b_i(Y)$. Note
that, in general, $\beta_i(Y)$ can be negative: for instance,
$$P(\mathbb{G}_{m,k};T)=P(\mathbb{P}^1_k;T)-P(\{0\};T)-P(\{\infty\};T)=T^2-1$$
For $k=\C$,
$$\beta_i(Y)=\sum_{j\geq 0}(-1)^{i+j} \mathrm{dim}\,Gr_i^W H^j_{c}(Y(\C),\Q)$$

The invariants $P(Y;T)$ and $\beta_i(Y)$ not only depend on the
scheme $Y$ but also on the base field $k$; if we want to make this
explicit, we'll write $P(f;T)$ and $\beta_i(f)$, with
$f:Y\rightarrow \Spec k$ the structural morphism.
\subsection{Finite base field.}\label{sec-poinfin} We can also define a Poincar\'e polynomial
for a finite base field $k$, using Deligne's theory of weights.
Denote by $q$ the cardinality of $k$. Recall that, for any integer
$w\geq 0$, a Weil number of weight $w$ (w.r.t. $q$) is an
algebraic integer $\alpha$ such that $|i(\alpha)|=q^{w/2}$ for
each embedding $i:\Q(\alpha)\hookrightarrow \C$. A fundamental
result by Deligne \cite[3.3.4]{deligne-weilII} says the following:
if $X$ is a separated $k$-scheme of finite type, and $\ell$ a
prime invertible in $k$, then for any integer $i\geq 0$, each
eigenvalue $\alpha$ of the geometric Frobenius on $H^i_c(X\times_k
k^s,\Q_\ell)$ is a Weil number, and its weight $w(\alpha)$ is
contained in $\{0,\ldots,i\}$. Moreover, if $X$ is proper and
smooth over $k$, then $w(\alpha)=i$ (``purity of weight''
\cite[3.3.5]{deligne-weilII}).

\begin{definition}
Assume that $k$ is finite. For any separated $k$-scheme of finite
type $X$ and each pair of integers $i,j\geq 0$, we define
$\beta^j_i(X)$ as the number of weight $i$ eigenvalues (counted
with multiplicities) of the geometric Frobenius on
$H^j_{c}(X\times_k k^s,\Q_\ell)$. We put $\beta_i(X)=\sum_{j\geq
0}(-1)^{i+j}\beta^j_i(X)$, and we call this integer the $i$-th
virtual Betti number of $X$. We define the Poincar\'e polynomial
$P(X;T)$ of $X$ by
$$P(X;T)=\sum_{i\geq 0}(-1)^i\beta_i(X)T^i$$
\end{definition}
 The virtual Betti numbers
$\beta_i(X)$, and hence the Poincar\'e polynomial $P(X;T)$, are
independent of $\ell$~: as noted in
\cite[p.\,28\,(2b)]{katz-intro}, $(-1)^{i+1}\beta_i(X)$ is the
degree of the ``weight $i$ part'' of the zeta function of $X$
(beware that Katz' definition of virtual Betti number differs from
ours by a factor $(-1)^i$). By purity of weight,
$b_i(X)=\beta^i_i(X)=\beta_i(X)$ if $X$ is proper and smooth over
$k$. The invariants $P(X;T)$ and $\beta_i(X)$ not only depend on
the scheme $X$, but also on the base field $k$. If we want to make
the base field explicit, we'll write $P(f;T)$ and $\beta_i(f)$
instead, with $f:X\rightarrow \Spec k$ the structural morphism.

\begin{lemma}[Additivity and multiplicativity]\label{Paddmult}
Assume that $k$ is finite. There exists a unique ring morphism
$$P:K_0(Var_k)\rightarrow \Z[T]$$ which maps the class $[X]$ of
any $k$-variety $X$ to the Poincar\'e polynomial $P(X;T)$.
\end{lemma}
\begin{proof}
Uniqueness is obvious. Well-definedness and additivity follow
immediately from the excision long exact sequence;
multiplicativity from the K\"unneth formula. Alternatively,
$P(X;T)$ can be computed from the \'etale realization
$\acute{e}t(X)$.
\end{proof}

As noted above, the Poincar\'e polynomial still has the property
$$P(X;T)=\sum_{i\geq 0}(-1)^ib_i(X)T^i$$
for any proper and smooth $k$-variety $X$, by purity of weight;
however, it is not clear if this property uniquely defines the
morphism $P:K_0(Var_k)\rightarrow \Z[T]$ (unless we assume the
existence of resolution of singularities for $k$-varieties).

\subsection{Base field of characteristic $p>0$.}\label{sec-poinp}
Let $X$ be a Noetherian scheme, and let $B$ be a set. We denote by
$X^o$ the set of closed points of $X$. We say that a function
$a:X^o\rightarrow B$ is constructible, if there exists a
stratification $\mathscr{S}$ of $X$ into locally closed subsets,
such that $a$ is constant on $S\cap X^o$ for each member $S$ of
$\mathscr{S}$. Likewise, we say that a function $b:X\rightarrow B$
is constructible if there exists a stratification $\mathscr{T}$ of
$X$ into locally closed subsets, such that $b$ is constant on $T$
for each member $T$ of $\mathscr{T}$.  We denote by
$\mathcal{C}(X,B)$, resp. $\mathcal{C}(X^o,B)$, the ring of
constructible functions on $X$, resp. $X^o$, with values in $B$.

If $X$ is a Jacobson scheme (e.g. of finite type over a field, or
over $\Z$) then any constructible function $a:X^o\rightarrow B$
extends uniquely to a constructible function $a:X\rightarrow B$.

\begin{prop}\label{Pgen}
Let $k$ be a finite field. For any separated $k$-scheme of finite
type $X$,
 there exists a unique ring morphism
$$P:K_0(Var_X)\rightarrow \mathcal{C}(X,\Z[T])$$
such that $P([Y])(x)=P(f_x;T)$ for every separated morphism of
finite type $f:Y\rightarrow X$ and every closed point $x$ of $X$.
Here $f_x:Y\times_X x\rightarrow \Spec k(x)$ is the morphism
obtained from $f$ by base change.

If $g:X'\rightarrow X$ is a morphism of separated $k$-schemes of
finite type, then the diagram
$$\begin{CD}
K_0(Var_X)@>>> K_0(Var_{X'})
\\ @VPVV @VVPV
\\ \mathcal{C}(X,\Z[T]) @>(\cdot)\circ g>> \mathcal{C}(X',\Z[T])
\end{CD}$$
commutes (the horizontal arrows are the natural base change
morphisms).
\end{prop}
\begin{proof}
Uniqueness of $P$ is obvious, since an element of
$\mathcal{C}(X,\Z[T])$ is determined by its values on $X^o$. To
prove its existence, first note that the function $x\mapsto
P(f_x;T)$ is constructible on $X^o$ since the sheaves
$R^if_{!}(\Q_\ell)$ are mixed \cite[3.3.1]{deligne-weilII}. Hence,
this function extends uniquely to a function $P(f;T)$ in
$\mathcal{C}(X,\Z[T])$. The invariant $P(\cdot\,;T)$ satisfies the
scissor relations in $K_0(Var_X)$~: since the property of being a
closed (resp. open) immersion is stable under base change, we can
reduce to the case where $X$ is a point, which was proven in Lemma
\ref{Paddmult}.

 Commutativity of the base change diagram is also
immediately reduced to the case where $X$ and $X'$ are points;
this case is clear from the definition of the virtual Betti
numbers.
\end{proof}

\begin{cor}
Let $k$ be any field of characteristic $p>0$. Using the notation
in Section \ref{sec-spread}, there exists a unique ring morphism
$$P:K_0(Var_k)\rightarrow \Z[T]$$ such that, for any object $A$ of
$\mathscr{A}_k$ and any separated $A$-scheme of finite type
$f:X\rightarrow \Spec A$,
$$(P\circ \phi)([X])=P(f;T)(\eta)\quad \in \Z[T]$$
where $\eta$ is the generic point of $\Spec A$.
\end{cor}
\begin{proof}
This follows from Proposition \ref{spread} and Proposition
\ref{Pgen}.
\end{proof}

\begin{definition}
For any field $k$ of characteristic $p>0$ and any separated
$k$-scheme of finite type $X$, we define the Poincar\'e polynomial
$P(X;T)$ of $X$ as the image of $[X]$ under the ring morphism
$$P:K_0(Var_k)\rightarrow \Z[T]$$
Writing
$$P(X;T)=\sum_{i\geq 0}(-1)^i\beta_i(X)T^i$$
we call $\beta_i(X)\in \Z$ the $i$-th virtual Betti number of $X$.

If we want to make the base field explicit, we write $P(f;T)$ and
$\beta_i(f)$, with $f:X\rightarrow \Spec k$ the structural
morphism.
\end{definition}
Note that the definition of $P(X;T)$ and $\beta_i(X)$ does not
require the choice of a prime $\ell$ (since the definition over
finite fields is independent of $\ell$).

\begin{remark}
We should point out that, if $k$ is finitely generated, the
Poincar\'e polynomial can also be realized as the composition of
the realization
$$\mu_k:K_0(Var_k)\rightarrow K_0(Rep_{G_k}\Q_\ell)[T]:[X]\mapsto
\sum_{i\geq 0}\left(\sum_{j\geq 0}(-1)^j [Gr^W_i H^j_c(X\times_k
k^s,\Q_\ell)]\right) T^i$$ from \cite{naumann} with the forgetful
ring morphism
$$K_0(Rep_{G_k}\Q_\ell)[T]\rightarrow K_0(\Q_\ell)[T]\cong \Z[T]$$
\end{remark}

%
\subsection{Arbitrary base field}
\begin{prop}\label{poin-euler}
For any field $k$ and any separated $k$-scheme of finite type $X$,
$P(X;1)=\chi_{top}(X)$.
\end{prop}
\begin{proof}
If $k$ has characteristic zero, this follows from the fact that
the equality holds for smooth and proper $k$-varieties, since
their isomorphism classes generate $K_0(Var_k)$ by Hironaka's
resolution of singularities. If $k$ is finite, it follows
immediately from the definition. If $k$ is any field of
characteristic $p>0$, it follows from the finite field case and
the fact that for any object $A\in \mathscr{A}_k$ and any
separated morphism of finite type $f:X\rightarrow \Spec A$, the
function
$$\Spec A\rightarrow \Z:x\mapsto \chi_{top}(X\times_{\Spec A}x)$$
is constructible, by constructibility of the sheaves
$R^if_{!}(\Q_\ell)$ and proper base change \cite[VI(3.2)]{Milne}.
\end{proof}

\begin{prop}\label{components}
Let $k$ be any field, and $X$ a separated $k$-scheme of finite
type, of dimension $n$. Then the Poincar\'e polynomial $P(X;T)$
has degree $2n$, and the coefficient $\beta_{2n}(X)$ of $T^{2n}$
is equal to the number of irreducible components of dimension $n$
of $X\times_k k^{s}$.
\end{prop}
\begin{proof}
We may assume that $X$ is reduced. Passing to a finite separable
extension of $k$, we may assume that the irreducible components of
$X$ are geometrically irreducible. Now we proceed by induction on
$n$. If $n=0$, then the statement is clear, so assume that we have
proven the result for varieties of dimension $<n$ over any field.
Then taking away closed subvarieties from $X$ of dimension $<n$
does not change the value of $\beta_{i}(X)$ for $i\geq 2n$, so we
may as well assume that the connected components of $X$ are
geometrically irreducible. By additivity, it suffices to consider
the case where $X$ itself is geometrically irreducible.

First, assume that $k$ has characteristic zero. We may suppose
that $k$ is algebraically closed. The class $[X]$ of $X$ in
$K_0(Var_k)$ can be written as the class $[Y]$ of a smooth,
proper, irreducible $k$-variety $Y$ plus a $\Z$-linear combination
of classes $[Z_i]$ of $k$-varieties $Z_i$ of dimension $<n$, by
 Hironaka's resolution of
singularities. Hence, by the induction hypothesis, the Poincar\'e
polynomial of $X$ has degree at most $2n$, and
$\beta_{2n}(X)=b_{2n}(Y)=1$.

Now assume that $k$ has characteristic $p>0$.
 There exist an object $A\in \mathscr{A}_k$ and a model
$X'$ for $X$ over $A$; by \cite[9.7.7]{ega4.3}, we may assume that
$X\times_{\Spec A} x$ is geometrically irreducible for each closed
point $x$ of $\Spec A$. By definition of the Poincar\'e
polynomial, we may suppose that $k$ is finite. Then the Poincar\'e
polynomial of $X$ has degree at most $2n$, and $Gr^W_{2n}
H_c^j(X\times_k k^s,\Q_\ell)$ vanishes for $j\neq 2n$, because
 $R^if_{!}(\Q_\ell)$ is mixed of
weight $\leq i$ by \cite[3.3.1]{deligne-weilII}. Hence,
$$\beta_{2n}(X)=\mathrm{dim}\,Gr^W_{2n}H^{2n}_c(X\times_k k^s,\Q_\ell)$$
Moreover, by \cite[VI(11.3)]{Milne} there exists a
Galois-equivariant isomorphism
$$H^{2n}_c(X\times_k k^s,\Q_\ell)\rightarrow \Q_\ell(-n)$$ so $H^{2n}_c(X\times_k
k^s,\Q_\ell)$ has pure weight $2n$ and $\beta_{2n}(X)=1$.
\end{proof}

\subsection{Arbitrary base scheme}
\begin{definition}
For any separated morphism of finite type $f:Y\rightarrow X$ in
$(Sch)$, we denote by $P(f;T)$ the function
$$P(f;T):X\rightarrow \Z[T]:x\mapsto P(f_x;T)$$ where
$f_x:Y\times_X x\rightarrow \Spec k(x)$ is the morphism obtained
by base change. We call $P(f;T)$ the Poincar\'e ploynomial of $f$.
Writing $P(f;T)$ as
$$\sum_{i\geq 0}(-1)^i\beta_i(f)T^i$$ we call the function
$\beta_i(f):X\rightarrow \Z$ the $i$-th virtual Betti number of
$f$.
\end{definition}
\begin{lemma}[Base Change]\label{bchange-2}
Let $g:X'\rightarrow X$ be a morphism of schemes, and let
$f:Y\rightarrow X$ be a separated morphism of finite type. If we
denote by $f':Y\times_X X'\rightarrow X$ the morphism obtained
from $f$ by base change, then $P(f';T)=P(f;T)\circ g$.
\end{lemma}
\begin{proof}
It suffices to consider the case where $X=\Spec k$ and $X'=\Spec
k'$ with $k\subset k'$ fields. If $k$ has characteristic zero, the
result follows from the fact that the $\ell$-adic Betti numbers
are invariant under extension of the base field
\cite[VI(4.3)]{Milne}. If $k$ has characteristic $p>0$, it
suffices to note that the diagram
$$\begin{CD}
K_0(Var_A)@>\phi^k_A >> K_0(Var_{k}) \\@V\phi^{k'}_AVV @VVPV
\\  K_0(Var_{k'}) @>>P> \Z[T]
\end{CD}$$
commutes for each object $A$ of $\mathscr{A}_k$ (both paths from
$K_0(Var_A)$ to $\Z[T]$ coincide with the morphism
$P(\cdot)(\eta)$ with $\eta$ the generic point of $\Spec A$).
\end{proof}

\begin{prop}\label{smoothbetti}
Let $X$ be a locally Noetherian scheme, and let $f:Y\rightarrow X$
be a smooth and proper morphism.
 Then $P(f;T)$ is locally
constant, and for any point $x$ of $X$ and any integer $i\geq 0$,
$\beta_i(f)(x)=b_i(f_x)$ where $f_x:Y\times_X x\rightarrow x$ is
the morphism obtained from $f$ by base change.
\end{prop}
\begin{proof}
By definition, $\beta_i(f)(x)=\beta_i(f_x)$ for each $i\geq 0$. If
$k(x)$ has characteristic zero, then $\beta_i(f_x)=b_i(f_x)$ by
definition; if $k(x)$ is finite, the same holds by purity of
weight. If $k(x)$ has characteristic $p>0$, we can always find an
object $A$ of $\mathscr{A}_{k(x)}$ and a smooth and proper
$A$-model $h:Z\rightarrow \Spec A$ for $f_x$ by
\cite[8.10.5]{ega4.3} and \cite[17.7.8]{ega4.4}. By definition,
$P(f_x;T)=P(h;T)(\eta)$ where $\eta$ is the generic point of
$\Spec A$.

For any point $y$ of $\Spec A$, we denote by $h_y:Y\times_A
k(y)\rightarrow \Spec k(y)$ the morphism obtained by base change.
If $y$ is closed, then $k(y)$ is finite, and since $h$ is smooth
and proper, $\beta_i(h)(y)=b_i(h_y)$. However, both sides of the
equality are constructible as functions in $y\in \Spec A$\,: for
the left hand side this follows from Proposition \ref{Pgen}, and
for the right hand side by applying proper base change to the
lisse sheaf $R^ih_*(\Q_\ell)$ for any prime $\ell$ invertible in
$k(x)$ \cite[VI(2.3+4.2)]{Milne}. Hence,
$$\beta_i(f_x)=\beta_i(h)(\eta)=b_i(h_\eta)=b_i(f_x)$$ (the last
equality follows from invariance of $\ell$-adic Betti numbers
under extension of the base field \cite[VI(4.3)]{Milne}).

Finally, the fact that $P(f;T)$ is locally constant follows from
the fact that the function $x\mapsto b_i(f_x)$ is locally constant
on $X$~: we may assume that there exists a prime $\ell$ invertible
on $X$, and we apply proper base change to the lisse sheaf
$R^if_*(\Q_\ell)$.
\end{proof}

\begin{prop}[Constructibility]\label{conspropo}
Let $X$ be a Noetherian scheme. For any separated morphism of
finite type $f:Y\rightarrow X$,
 the map
$$P(f;T):X\rightarrow \Z[T]$$ is constructible.
\end{prop}
\begin{proof}
By Noetherian induction, it suffices to find a non-empty open
subscheme $U$ of $X$ such that $P(f;T)$ if constant on $U$, so we
may assume that $X$ is integral and affine, say $X=\Spec B$, and
that there exists a prime $\ell$ invertible on $X$. By the
canonical isomorphism $(Y_{red}\times_X x)_{red}\cong (Y\times_X
x)_{red}$ for any point $x$ of $X$, we may suppose that $Y$ is
reduced.

By \cite[8.8.2]{ega4.3} there exists a finitely generated
sub-$\Z[1/\ell]$-algebra $C$ of $B$, and a reduced separated
$C$-scheme of finite type $Y'$, such that $Y$ is isomorphic to
$Y'\times_C B$ over $B$. By Lemma \ref{bchange-2}, we may assume
that $B=C$ and $Y=Y'$. Then the sheaves $R^if_{!}(\Q_\ell)$ are
mixed \cite[3.3.1]{deligne-weilII}, so there exists a non-empty
open subset $U$ of $X$ such that $P(f_x;T)=P(f_y;T)$ for any pair
of \textit{closed} points $x,y$ on $X$. By definition of the
Poincar\'e polynomial, this implies that $P(f_x;T)=P(f_y;T)$ for
any pair of points $x,y$ of $U$ which lie over a closed point of
$\Spec \Z[1/\ell]$.

Hence, we may assume that the generic point $\eta$ of $X$ lies
over the generic point of $\Spec \Z[1/\ell]$.
We proceed by induction on the
 dimension $n$ of $Y_\eta=Y\times_X \eta$.

If $Y_\eta$ is empty, then there exists an open neighbourhood $V$
of $\eta$ in $X$ such that the fibers of $f$ over $V$ are empty
\cite[9.2.6]{ega4.3}, hence $P(f;T)=0$ on $V$. So assume that
$n\geq 0$ and that the result has been proven for morphisms for
which the dimension of the generic fiber is $<n$. Let
$\overline{f}:\overline{Y}\rightarrow X$ be a compactification of
the morphism $f$ (i.e. $\overline{f}$ is proper and there exists a
dense open immersion $j:Y\rightarrow \overline{Y}$ with
$f=\overline{f}\circ j$). Denote by $\partial \overline{Y}$ the
complement of $Y$ in $\overline{Y}$ (with its reduced closed
subscheme structure). Then $\partial \overline{Y}\times_X \eta$
has dimension $<n$,
so by the induction hypothesis and additivity of the Poincar\'e
polynomial, we may as well assume that $Y=\overline{Y}$, i.e. that
$f$ is proper.

 Since $k(\eta)$ has characteristic
zero, and $Y_\eta$ is reduced, there exists a proper birational
morphism of $k(\eta)$-varieties $h':Z'\rightarrow Y_\eta$ such
that $Z'$ is proper and smooth over $k(\eta)$. Shrinking $X$, we
may suppose that $h$ is obtained by base change from a proper
birational morphism of $X$-varieties $h:Z\rightarrow Y$ with $Z$
smooth and proper over $X$, by \cite[8.8.2+9.6.1]{ega4.3} and
\cite[17.7.11]{ega4.4}. Then we can find open subschemes $U$ and
$V$ of $Z$, resp. $Y$, such that $h$ restricts to an isomorphism
$U\cong V$, and such that $(Z-U)\times_X \eta$ and $(Y-V)\times_X
\eta$ have dimension $<n$. By additivity and our induction
hypothesis, it suffices to prove the result for the proper and
smooth morphism $Z\rightarrow X$; this case was settled in
Proposition \ref{smoothbetti}.
\end{proof}

\begin{theorem}\label{thm-poin}
The Poincar\'e polynomial $P(\cdot,T)$ is the unique invariant
which associates to every separated morphism of finite type
$f:Y\rightarrow X$ in $(Sch)$ a function $P(f;T):X\rightarrow
\Z[T]$ with the following properties:
\begin{enumerate}
\item \textit{Constructibility:} \label{propcons} If $X$ is
Noetherian, then $P(f;T)$ is constructible. \item
\label{propbc}\textit{Base change:} If $g:X'\rightarrow X$ is a
morphism in $(Sch)$ and $f':Y\times_X X'\rightarrow X'$ is the
morphism obtained by base change, then
$$P(f';T)=P(f;T)\circ g$$
\item \label{propnorm} If $X=\Spec k$ with $k$ a finite field,
then $P(f;T)$ is the Poincar\'e polynomial defined in Section
\ref{sec-poinfin}.
\end{enumerate}
If $X$ is a Noetherian scheme, then there exists a unique ring
morphism
$$P(\cdot\,;T):K_0(Var_X)\rightarrow \mathcal{C}(X,\Z[T])$$
such that $P([Y];T)=P(f;T)$ for any separated $X$-scheme of finite
type $f:Y\rightarrow X$. If $g:X'\rightarrow X$ is a morphism of
Noetherian schemes, then the square
$$\begin{CD}
K_0(Var_X)@>>> K_0(Var_{X'}) \\@VP(\cdot,T)VV @VVP(\cdot,T) V
\\ \mathcal{C}(X,\Z[T])@>(\cdot)\circ g>> \mathcal{C}(X',\Z[T])
\end{CD}$$
commutes (the horizontal arrows are the natural base change
morphisms).
\end{theorem}
\begin{proof}
We proved in Lemma \ref{bchange-2} and Proposition \ref{conspropo}
that the Poincar\'e polynomial satisfies $(\ref{propcons})$ and
$(\ref{propbc})$, and $(\ref{propnorm})$ is clear by definition.

Let us show that such an invariant is unique. By $(\ref{propbc})$,
$P(\cdot;T)$ is uniquely determined by its values on morphisms
$f:Y\rightarrow X$ with $X=\Spec k$ and $k$ a field. If
$f:X\rightarrow \Spec k$ a seperated morphism of finite type, we
can find a finitely generated sub-$\Z$-algebra $C$ of $k$ and a
separated morphism of finite type $h:Z\rightarrow \Spec C$ such
that $X$ is $k$-isomorphic to $Z\times_C k$, by
\cite[8.8.2]{ega4.3}. Then
$$P(f;T)=P(h;T)(\eta)$$ with $\eta$ the generic point
of $\Spec C$, by $(\ref{propbc})$. The function $P(h;T)$ is a
constructible function, by $(\ref{propcons})$, so it is uniquely
determined by its values on the closed points of $\Spec C$, which
have finite residue field. Hence, $P(h;T)$ is uniquely determined,
by $(\ref{propbc})$ and $(\ref{propnorm})$.

It only remains to show that $P(\cdot,T)$ satisfies the scissor
relations in $K_0(Var_X)$, if $X$ is a Noetherian scheme. Since
the property of being a closed (resp. open) immersion is stable
under base change, we can reduce to the case where $X$ is a point;
this case is clear from Lemma \ref{Paddmult} and the definition of
the Poincar\'e polynomial.
\end{proof}

\bibliographystyle{hplain}
\bibliography{wanbib,wanbib2}
\end{document}